\documentclass[a4paper, 12pt]{amsart}

\usepackage{amssymb}
\usepackage{amsthm}
\usepackage{amsmath}
\usepackage{mathrsfs}
\usepackage{comment}
\usepackage{longtable}
\usepackage{tabu}
\usepackage{tabularx}
\usepackage{ltablex}

\makeatletter
\@namedef{subjclassname@2020}{%
  \textup{2020} Mathematics Subject Classification}
\makeatother

\theoremstyle{plain}
\newtheorem{thm}{Theorem}[section]		
\newtheorem{prop}[thm]{Proposition}
\newtheorem{cor}[thm]{Corollary}
\newtheorem{lem}[thm]{Lemma}

\theoremstyle{definition}

\newtheorem{df}{Definition}[section]

\theoremstyle{remark}
\newtheorem{rmk}{Remark}[section]
\newtheorem*{ac}{Acknowledgements}

\newcommand{\To}{\Longrightarrow}

\newcommand{\nn}{\mathbb{N}}
\newcommand{\zz}{\mathbb{Z}}
\newcommand{\qq}{\mathbb{Q}}
\newcommand{\rr}{\mathbb{R}}

\newcommand{\elel}{\mathrm{L}}

\renewcommand{\labelenumi}{{\rm (\arabic{enumi})}}
\DeclareMathOperator{\card}{card}

\DeclareMathOperator{\cl}{CL}

\DeclareMathOperator{\pc}{\mathscr{P}}
\DeclareMathOperator{\nai}{INT}

\DeclareMathOperator{\met}{M}

\DeclareMathOperator{\seq}{Seq}

\DeclareMathOperator{\clo}{\mathcal{C}}

\DeclareMathOperator{\ult}{UM}
\DeclareMathOperator{\cu}{CUM}
\DeclareMathOperator{\map}{Map}

\DeclareMathOperator{\coi}{coi}
\makeatletter
\@addtoreset{equation}{section}

\makeatother

\begin{document}

\title[An Embedding, An Extension,  and  An Interpolation]
{An Embedding, An Extension,  and  An Interpolation of Ultrametrics}
\author[Yoshito Ishiki]
{Yoshito Ishiki}
\address[Yoshito Ishiki]
{\endgraf
Graduate School of Pure and Applied Sciences
\endgraf
University of Tsukuba
\endgraf
Tennodai 1-1-1, Tsukuba, Ibaraki, 305-8571, Japan}
\email{ishiki@math.tsukuba.ac.jp}

\date{\today}
\subjclass[2020]{Primary 54E35; Secondary 30L05, 54E52}
\keywords{Ultrametric space, Baire space, Isometric embedding}
\thanks{The author was supported by JSPS KAKENHI Grant Number 18J21300.}

\begin{abstract}
The notion of  ultrametrics can be  considered as a zero-dimensional  analogue of ordinary metrics, 
and  it is expected to prove  ultrametric versions of theorems on metric spaces. 
In this  paper, 
we provide ultrametric versions of 
the Arens--Eells isometric embedding theorem of metric spaces, 
the Hausdorff extension theorem of metrics, 
the Niemytzki--Tychonoff characterization  theorem of the compactness, 
and the author's  interpolation theorem of metrics and theorems on dense subsets of spaces of metrics. 
\end{abstract}
\maketitle

\section{Introduction}\label{intro}

Let $X$ be a  set. 
A metric $d$ on $X$ is said to be an \emph{ultrametric} or  
a \emph{non-Archimedean metric}
if for all $x, y, z\in X$ we have 
\begin{align}\label{al:ult}
d(x, y)\le d(x, z)\lor d(z, y), 
\end{align}
where the symbol 
$\lor $ 
stands for the maximum operator on $\rr$. 
The inequality (\ref{al:ult}) 
is called the \emph{strong triangle inequality}.  
We say that a set 
$S$  
is a \emph{range set} 
if  
$S\subseteq [0, \infty)$ 
and 
$0\in S$. 
For a range set 
$S$, 
we say that a metric 
$d: X^2\to [0, \infty)$ 
on 
$X$ 
is 
\emph{$S$-valued} 
if 
$d(X^2)\subseteq S$. 
Note that 
$[0, \infty)$-valued ultrametrics are nothing but ultrametrics. 
The $S$-valued ultrametrics are studied as a special case and a reasonable restriction of ultrametrics. 
Gao and Shao \cite{GS11} studied 
$R$-valued Urysohn  universal ultrametric spaces 
for a countable range  set $R$. 
Brodskiy, Dydak, Higes and Mitra \cite{BDHM} utilized 
$(\{0\}\cup \{\, 3^{n}\mid n\in \zz\, \})$-valued ultrametrics 
for their study on the $0$-dimensionality in categories of metric spaces. 
\par

The notion of  ultrametrics can be considered as 
a $0$-dimensional  analogue of ordinary metrics, 
and  it is often expected to prove ultrametric versions of theorems 
on metric spaces. 
In this paper, 
for every range set $S$, 
we provide $S$-valued ultrametric versions of 
the Arens--Eells isometric embedding theorem \cite{AE} of metric spaces, 
the Hausdorff extension theorem \cite{Haus1} of metrics, 
the Niemytzki--Tychonoff characterization \cite{NT} of the compactness, 
and the author's  interpolation theorem of metrics and theorems 
on dense $G_{\delta}$ subsets of spaces of metrics
\cite{IsI}. 
\par
%%%%%%%%%%%MMM
\begin{rmk}\label{rmk:ps}
For every range set $S$, 
there exists an $S$-valued ultrametric space $(X, d)$ such that 
$d(X^2)=S$. 
Define an ultrametric $d$ on $S$ by 
\[
d(x, y)=
\begin{cases}
x\lor y & \text{if $x\neq y$}\\
0  & \text{if $x=y$}. 
\end{cases}
\]
Then we have $d(S^2)=S$, and hence $(S, d)$ is a desired space. 
This construction  was given by  
Delhomm\'{e}, Laflamme, Pouzet and Sauer \cite[Proposition 2]{DLPS} (see also \cite{DDP}). 
We use their ultrametrics in Lemmas \ref{lem:ultcb} and   \ref{lem:u-atai}. 
\end{rmk}
%%%%%%%%%%%%%%%%%MMM
In 1956,  
Arens and Eells \cite{AE} established  the result which today we call the 
Arens--Eells embedding theorem, 
stating that 
for  every  metric space $(X, d)$, 
there exist a real normed linear space  
$V$ 
and 
an isometric embedding 
$I: X\to V$ 
such that 
\begin{enumerate}
\item $I(X)$ is closed in $V$;
\item $I(X)$ is  linearly independent in $V$.
\end{enumerate}
Before stating  our first main result, 
we introduce some basic notions. 
Let $R$ be a commutative ring, 
and 
let $V$ be an $R$-module. 
A subset $S$ of $V$ is said to be 
\emph{$R$-independent} 
if 
for every finite subset 
$\{f_1, \dots,  f_n\}$ of $S$, 
and for all  
$N_1, \dots, N_n\in R$, 
the identity  
$\sum_{i=1}^nN_if_i=0_V$ 
implies 
$N_i=0_R$ for all $i$, 
where $0_V$ and $0_R$ stand for 
the zero elements of $V$ and $R$, respectively. 
%%%%%%%%%%%%%%%%%%
A function 
$\|*\|: V\to [0, \infty)$ 
is said to be an 
\emph{ultra-norm on $V$} 
if the following are satisfied:
\begin{enumerate}
\item $\|x\|=0$ if and only if $x=0_V$;
\item for every $x\in V$, we have $\|-x\|=\|x\|$;
\item for all $x, y\in V$, we have $\|x+y\|\le \|x\|\lor \|y\|$. 
\end{enumerate}
The pair $(V, \|*\|)$ is called an 
\emph{ultra-normed $R$-module} 
(see also \cite{Wa}). 
\par

Megrelishvili and Shlossberg \cite{MS2} 
have already proven an ultrametric version 
of the Arens--Eells embedding theorem, 
stating that  every ultrametric space is 
isometrically embeddable into an ultra-normed Boolean group 
(a $\zz/2\zz$-module) as a closed set of it, 
which is a consequence of  their study 
on free non-Archimedean topological groups 
and Boolean groups with actions from topological groups. 
By introducing module structures into  universal  
ultrametric spaces of Lemin--Lemin type \cite{LL}, 
we obtain 
a more general $S$-valued  ultrametric version of the Arens--Eells embedding theorem. 
\begin{thm}\label{thm:ultemb}
Let $S$ be a range set possessing at least two elements. 
%%%%%%%%%%%%%MM%%%%%%%%%5
Let $R$ be an integral domain, and let  
$(X, d)$ be an $S$-valued ultrametric space. 
Then there exist an $S$-valued ultra-normed $R$-module  $(V, \|*\|)$,   
and an isometric  embedding $I: X\to V$
such that 
\begin{enumerate}
\item $I(X)$ is closed in $V$;
\item $I(X)$ is $R$-independent in $V$. 
\end{enumerate}
Moreover, 
if $(X, d)$ is complete, 
then we can choose $(V, \|*\|)$ as a complete  metric space.
\end{thm}
\begin{rmk}
Let $R$ be an integral domain. 
Let $t_R$ be the trivial valuation 
on $R$ defined by $t_R(x)=1$ 
if $x\neq 0_R$; otherwise $t_R(x)=0$. 
Let $(V, \|*\|)$ be a $R$-module constructed in the proof of Theorem \ref{thm:ultemb}. 
 Then the  ultra-norm $\|*\|$ on $V$ is compatible with $t_R$, 
 i.e., 
for every $r\in R$ and for every $x\in V$, 
we have $\|r\cdot x\|=t_R(r)\|x\|$.  
%%%%%%%%%%%%%%%%%%%%%5
For every finite field, 
there exist no valuations on it except the trivial valuation. 
Thus we can consider that 
Theorem \ref{thm:ultemb} includes the Arens--Eells embedding  theorem 
into normed spaces  over all finite fields. 
The author does not know whether 
such an embedding theorem into normed spaces  
over all non-Archimedean valued fields holds true or not. 
\end{rmk}

\begin{rmk}
In Theorem \ref{thm:ultemb}, 
the assumption that $S$ possesses  at least two elements is necessary. 
Indeed, the set $\{0\}$ is a range set. 
Let $R$ be an integral domain. 
 Let $V$ be the zero $R$-module; namely, $V$ consists of 
only the zero element,  and let $\|*\|$ be  the trivial ultra-norm on $V$. 
Then all $\{0\}$-valued ultra-normed $R$-modules are isometric to 
the $\{0\}$-valued ultrametric space $(V, \|*\|)$. 
Since $V$ has no $R$-independent subset, 
an isometric embedding theorem of Arens--Eells-type does not holds for a range set $\{0\}$. 
\end{rmk}

\begin{rmk}
There are various theorems on  isometric embeddings 
from an ultrametric space
into a metric space with algebraic structures. 
For instance, 
Schikhof \cite{S} constructed 
an isometric embedding 
from an ultrametric space into a non-Archimedean valued field. 
The existence of an isometric embedding 
from an ultrametric space into a Hilbert space was first proven 
by Timan and Vestfrid \cite{TV1} in a separable case, 
and  was proven 
by A. J.  Lemin \cite{L} in a general setting. 
The papers \cite{TV2}, \cite{V94} and \cite{F} also contain related results. 
\end{rmk}

For a range set $S$, and for a topological  space $X$, 
we denote by 
$\ult(X, S)$  (resp.~$\met(X)$)
 the set of all $S$-valued ultrametrics (resp.~metrics) on $X$ that generate 
the same topology as the original one of $X$. 
We denote by $\ult(X)$ the set $\ult(X, [0, \infty))$. 
We say that a topological space $X$ is 
\emph{$S$-valued ultrametrizable} 
(resp.~\emph{ultrametrizable}) if
$\ult(X, S)\neq \emptyset$ (resp.~$\ult(X)\neq \emptyset$). 
We  say that $X$ is \emph{completely $S$-valued ultrametrizable}
(resp.~\emph{completely  ultrametrizable}) if there exists a complete metric $d\in \ult(X, S)$ (resp.~$d\in \ult(X)$). 
\par

We say that a range set $S$ has 
 \emph{countable coinitiality} if 
there exists 
a  strictly decreasing sequence 
$\{r_i\}_{i\in \nn}$ in $S$ with 
$\lim_{i\to \infty}r_{i}=0$.

\begin{rmk}
For a range set $S$  with the countable coinitiality, and for a topological space $X$, 
it is worth clarifying a relation between the  ultrametrizability 
and the $S$-valued ultrametrizability. 
In Proposition \ref{prop:ultequi}, 
we show that these two properties are equivalent to each other. 
\end{rmk}

In 1930, 
Hausdorff \cite{Haus1} proved the extension theorem stating  
that for every metrizable space $X$,  
for every closed subset $A$ of $X$, 
and for every metric $e\in \met(A)$, 
there exists a metric $D\in \met(X)$ with 
$D|_{A^2}=e$. 
\par

By using the Arens--Eells embedding theorem,  
Toru\'nczyk \cite{Tor} provided 
a simple proof of the Hausdorff extension theorem. 
Due to Toru\'nczyk's method and 
 Theorem \ref{thm:ultemb}, we can  
 prove an $S$-valued  ultrametric version  of the Hausdorff extension theorem. 
 
\begin{thm}\label{thm:ultex}
Let $S$ be a range set. 
Let $X$ be an $S$-valued  ultrametrizable space, 
and let $A$ be a closed subset of $X$. 
Then 
For every $e\in \ult(A, S)$, 
there exists  $D\in \ult(X, S)$ with 
$D|_{A^2}=e$. 
Moreover, 
if $X$ is completely metrizable and $e\in \ult(A, S)$ is complete, 
then we can choose $D$ as a complete $S$-valued ultrametric. 
\end{thm}

\begin{rmk}
There are several studies on extending a partial or continuous ultrametrics 
(see \cite{El}, \cite{S}, \cite{StT}, or  \cite{TZ}). 
The extension of ultrametrics defined on subsets of sets  to ultrametrics on whole 
sets can be considered as a special case of extension of a weight 
from the edge set of a graph to an ultrametric on the vertex set of this graph.
Some results in this direction  can be found in \cite{DP},  and it is an ultrametric version of the corresponding theorems on metrization of weighted graphs in \cite{DMV}.  
\end{rmk}

In 1928, 
Niemytzki and Tychonoff \cite{NT} proved that a metrizable space $X$ is 
compact if and only if all metrics in  $\met(X)$  are complete. 
Hausdorff \cite{Haus1} gave a simple proof 
of their characterization theorem 
by applying his extension theorem of metrics mentioned above. 
By using Hausdorff's argument and Theorem \ref{thm:ultex}, 
we obtain an ultrametric version  of the Niemytzki--Tychonoff theorem. 
\begin{cor}\label{cor:ultcpt}
Let $S$ be a range set with the countable coinitiality. 
Let $X$ be an $S$-valued ultrametrizable space. 
Then the space $X$ is compact  if and only if 
  for every ultrametric $d\in \ult(X, S)$  is complete. 
\end{cor}
%%%%%%%%%%%MMM
\begin{rmk}
In Corollary \ref{cor:ultcpt},   the assumption that  $S$ has countable  coinitiality is necessary.
 Indeed, let $S=\{0, 1\}$, and let $M$ be an infinite discrete space. 
 Let  $d$ be an   ultrametric  on $M$
such that $d(x, y)=1$ whenever $x\neq y$. 
 Then $\ult(M, \{0, 1\})=\{d\}$. 
 Thus $M$ is not compact, and every member of  $\ult(M, \{0, 1\})$ is complete. 
\end{rmk}
%%%%%%%%%%%%%MM

To state our  future  results, 
for a topological space $X$, 
and for a range set $S$, 
we define a function 
$\mathcal{UD}_X^S: \ult(X, S)^2\to [0, \infty]$
 by assigning $\mathcal{UD}_X^S(d, e)$ to  the infimum of 
 $\epsilon\in S\sqcup \{\infty\}$ such that 
 for all $x, y\in X$ we have 
 \[
 d(x, y)\le e(x, y)\lor \epsilon, 
 \]
 and 
 \[
  e(x, y)\le d(x, y)\lor \epsilon. 
 \]
The function 
$\mathcal{UD}_X^S$
 is an ultrametric on $\ult(X, S)$ valued in $\cl(S)\sqcup \{\infty\}$, 
 where $\cl(S)$ stands for the closure of $S$ in $[0, \infty)$. 
 
 \begin{rmk}
Let  $d, e\in \ult(X, S)$, and let $r\in S_+$. 
Then,  we have  $\mathcal{UD}_{X}^S(d, e)\le r$ if and only if  
  for all $x, y\in X$ satisfying either   $d(x, y)> r$ or  $e(x, y)>r$ we have  $d(x, y)=e(x, y)$. 
   \end{rmk}
 
 We also define a function 
$\mathcal{D}_X: \met(X)\times \met(X)\to [0, \infty]$
 by 
 \[
 \mathcal{D}_X(d, e)=\sup_{(x, y)\in X^2}|d(x, y)-e(x, y)|. 
 \]
The function 
$\mathcal{D}_X$
 is a metric on $\met(X)$ valued in $[0, \infty]$.

\begin{rmk}
Qiu \cite{Q} introduced the 
\emph{strong $\epsilon$-isometry} 
in the study on the  non-Archimedean Gromov--Hausdorff distance 
(see \cite{Z}). 
This concept  is 
an analogue for ultrametric spaces of the $\epsilon$-isometry 
in the study on the ordinary Gromov--Hausdorff distance (see \cite{BBI}). 
Roughly speaking, 
for a range set $S$, 
for  an $S$-valued  ultrametrizable space $X$, 
and for $S$-valued  ultrametrics $d, e\in \ult(X, S)$, 
the inequality  $\mathcal{UD}_X^S(d, e)\le \epsilon$  
is  equivalent to the statement that  
 the identity maps $1_X: (X, d)\to (X, e)$ and $1_X: (X, e)\to (X, d)$  
 are strong $\epsilon$-isometries. 
\end{rmk}

The author \cite{IsI} proved an interpolation theorem of metrics with 
the information of  $\mathcal{D}_X$ (see \cite[Theorem 1.1]{IsI}). 
As its application, 
the author  proved that 
for every non-discrete metrizable space $X$ 
the set of all metrics in $\met(X)$ 
with a transmissible property,  
which is a geometric property determined by finite subsets 
(see Definition \ref{def:transpro}), 
is dense $G_{\delta}$ in the metric space 
$(\met(X), \mathcal{D}_X)$ (see \cite[Theorem 1.2]{IsI}), 
and also proved a local version of it (see \cite[Theorem 1.3]{IsI}). 
\par

By using Theorems \ref{thm:ultemb} and  \ref{thm:ultex}, 
we can prove an ultrametric version of the author's interpolation theorem. 
\par

A family 
$\{H_i\}_{i\in I}$ 
of subsets of a topological space $X$ is said to be 
\emph{discrete} if for every 
$x\in X$
 there exists a neighborhood of $x$ 
 intersecting at most single member of $\{H_i\}_{i\in I}$. 
 For a range set $S$, and for a subset
  $E$ of $S$, 
  we denote by $\sup E$ 
  the supremum of $E$ taken  in $[0, \infty]$, 
  not in $S$. 
For $C\in [1, \infty)$, 
we say that $S$ is \emph{$C$-quasi-complete}
 if for every bounded subset $E$ of $S$, 
 there exists $s\in S$
with $\sup E\le s\le C\cdot\sup E$. 
We say that $S$ is \emph{quasi-complete} 
if $S$ is $C$-quasi-complete for some $C\in [1, \infty)$. 
Note that  a range set is  $1$-quasi-complete  if and only if 
it is closed under the supremum operator. 
\begin{thm}\label{thm:ultint}
Let $C\in [1, \infty)$, and 
let $S$ be a $C$-quasi-complete range set. 
Let $X$ be an $S$-valued %%%%%%MMM
 ultrametrizable space, and let 
$\{A_i\}_{i\in I}$ 
be a discrete family of closed subsets of $X$. 
Then for every $S$-valued  ultrametric  
$d\in \ult(X, S)$,   
and for every 
family 
$\{e_i\}_{i\in I}$ 
of ultrametrics with 
$e_i\in \ult(A_i, S)$ for all $i\in I$, 
 there exists an $S$-valued  ultrametric 
 $m\in \ult(X, S)$ 
 satisfying the following: 
\begin{enumerate}
\item for every $i\in I$ we have $m|_{A_i^2}=e_i$;
\item $\sup_{i\in I}\mathcal{UD}_{A_i}^S(e_{i}, d|_{A_i^2})
\le \mathcal{UD}_X^S(m, d)\le C\cdot \sup_{i\in I}\mathcal{UD}_{A_i}^S(e_{i}, d|_{A_i^2})$. 
\end{enumerate}
Moreover, 
if $X$ is completely metrizable,  
and if each $e_i\in \ult(A_i, S)$ is  complete, 
then we can choose $m\in \ult(X, S)$ as a complete metric.  
\end{thm}

Similarly to \cite{IsI}, 
by using Theorem \ref{thm:ultint}, we prove that 
for every   range set $S$ with the countable coinitiality, 
for every $S$-ultra-singular transmissible parameter $\mathfrak{G}$, and 
for every non-discrete ultrametrizable space $X$,  the set of all metrics $d\in \ult(X, S)$ for which 
$(X, d)$  does not satisfy the $\mathfrak{G}$-transmissible property is 
 a dense $G_{\delta}$ subset of $\ult(X, S)$ (see Theorem \ref{thm:ulttrans}), 
 and prove its local version (see Theorem \ref{thm:ultloc}). 
Note that Theorems \ref{thm:ulttrans} and \ref{thm:ultloc} are ultrametric analogues of Theorems 1.2 and   1.3 in \cite{IsI}, respectively.

For example, 
the doubling property for metric spaces is equivalent to 
a $\mathfrak{G}$-transmissible property for some 
ultra-singular transmissible parameter $\mathfrak{G}$ (see Subsection \ref{subsec:examples}). 
Note that the doubling property for metric spaces is equivalent to the finiteness of the Assouad dimension (see \cite[Chapter 10]{H}). 
Thus, by Theorem \ref{thm:ulttrans}, 
for every  a range set $S$ with the countable coinitiality, 
for every non-discrete ultrametrizable space $X$,
the set of all  metrics  $d\in \ult(X, S)$ for which $(X, d)$ has infinite Assouad dimension  is 
dense  $G_{\delta}$ in $\ult(X, S)$. 
\par

The organization of this paper is as follows: 
In Section \ref{sec:pre}, 
we review the basic or classical  statements 
on $S$-valued  ultrametric spaces and $0$-dimensional spaces. 
We give proofs of some of them. 
In Section \ref{sec:uae}, 
we observe that  a construction  of universal ultrametric spaces  
and isometric embeddings of Lemin--Lemin-type \cite{LL} 
can be applied  to $S$-valued ultrametric spaces for all range set $S$. 
We also discuss algebraic structures 
on  universal ultrametric spaces of Lemin--Lemin-type. 
After that, we prove Theorem \ref{thm:ultemb}. 
In Section \ref{sec:uh}, 
we prove Theorem \ref{thm:ultex} and 
Corollary \ref{cor:ultcpt} 
by following Toru\'nczyk and Hausdorff's methods, 
and by using Theorem \ref{thm:ultemb}. 
In Section \ref{sec:ui}, 
we prove Theorem \ref{thm:ultint} 
by converting  the author's proof of \cite[Theorem 1.1]{IsI} 
into an $S$-valued  ultrametric proof 
with Theorems \ref{thm:ultemb} and \ref{thm:ultex}. 
In Section \ref{sec:ulttrans}, 
we introduce transmissible property, 
originally defined in \cite{IsI}, 
and 
we prove Theorem \ref{thm:ulttrans} as  
an application of Theorem \ref{thm:ultint}.
In Section \ref{sec:ultloc}, 
we first show that  $\ult(X, S)$ is a Baire space for a range set $S$, 
and   
for a second countable locally compact  $X$ 
(see Lemma \ref{lem:ultBaire}). 
We next prove  Theorem \ref{thm:ultloc}, 
which is a local version of Theorem \ref{thm:ulttrans}. 
At the end of this paper, 
for the convenience for the readers, 
we add a table of notion and 
a table of 
symbols. 

%%%%%%%%%%%%%%%%%%%%%%%%%%%%%%%%%%%%%%%%%%%%%%%%%%%%%%%%%%%%%%%%%%%%%%%%%%%%%%%%%%%%%%%%%%%%%%%%%%%%%%%%%%%%%%%%%%%%%%%%%%%%%%%%%%%%%%%%%%%%%%%%%%%%%%%%%%%%%%%%%%%%%%%%%%%%%%%%%%%%%%%%%%%%%%%%%%%%%%%%%%%%%%%%%%%%%%%%%%%%%%%%%%%%%%%%%%%%%%%%%%%%%%%%%%%%%%%%%%%%%%%%%%%%%%%%%

\section{Preliminaries}\label{sec:pre}

In this paper,  
we denote by $\nn$ the set of all positive integers. 
For a set $E$, 
we denote by 
$\card(E)$ 
the cardinality of $E$. 
For a metric space $(X, d)$,  
and for a subset $A$ of $X$, 
we denote by 
$\delta_d(A)$ 
the diameter of $A$. 
We denote by 
$B(x, r)$ (resp.~$U(x, r)$)  
the closed (resp.~open)  ball centered at $x$ with radius $r$. 
We also denote by 
$B(A, r)$ 
the set $\bigcup_{a\in A}B(a,r)$. 
To emphasize metrics under consideration, 
we sometimes denote 
by 
$B(x, r; d)$ (resp.~$U(x, r; d)$) 
the closed (resp.~open) ball in $(X, d)$. 
For a range set $S$, 
we define 
$S_{+}=S\setminus \{0\}$. 

\subsection{Modification of ultrametrics}\label{modi:ult}
A function   $\psi: [0, \infty)\to [0, \infty)$  is said to be 
\emph{amenable} if $\psi^{-1}(\{0\})=\{0\}$. 
 Pongsriiam and  Termwuttipong  \cite{PT} proved the following 
 (see \cite[Theorem 9]{PT}, and  see also \cite[Theorem 2.9]{OD}):

\begin{thm}\label{thm:tp}
Let  $\psi: [0, \infty)\to [0, \infty)$ be a function. 
Then the following statements are equivalent:
\begin{enumerate}
\item  $\psi$ is increasing and amenable.
\item For every set $X$, and for every  ultrametric $d$ on $X$, 
the function $\psi\circ d$ is an  ultrametric on $X$. 
\item For every set $X$ with $\card(X)=3$, and for every  ultrametric $d$ on $X$, 
the function $\psi\circ d$ is an  ultrametric on $X$. 
\end{enumerate}
\end{thm}
\begin{rmk}
The condition $(3)$ in Theorem \ref{thm:tp} does not appear explicitly in the statement of  \cite[Theorem 9]{PT}; 
however, the proof of \cite[Theorem 9]{PT} contains it. 
\end{rmk}
The following lemma is an application  of Pongsriiam and  Termwuttipong's 
result for topologically compatible ultrametrics, 
and it can be considered as an ultrametric analogue of 
\cite[Theorem 3.2]{Cor}.
\begin{lem}\label{lem:ultpreserving}
Let $\psi: [0, \infty)\to [0, \infty)$ be a  function. 
Then the following statements are equivalent:
\begin{enumerate}
\item $\psi$ is   increasing,  amenable, and continuous at the point $0$.
\item For every topological space $X$, and for every $d\in \ult(X)$, 
we have $\psi\circ d\in \ult(X)$.  
\end{enumerate}
\end{lem}
\begin{proof}
We first prove the implication  $(1)\To (2)$. 
Take $d\in \ult(X)$. 
By Theorem \ref{thm:tp}, 
we see that $\psi\circ d$ is an ultrametric on $X$. 
We now prove that $\psi\circ d$ induces the same topology on $X$. 
Take 
$x\in X$ 
and 
$r\in (0, \infty)$.
Since $\psi$ is continuous at $0$, 
there exists $l\in (0, \infty)$ with $\psi(l)<r$. 
Then,  we have 
$U(x, l; d)\subseteq U(x, r; \psi\circ d)$. 
Since $\psi$ is increasing, 
we have 
$U(x, \psi(r)/2;\psi\circ d)\subseteq U(x, r; d)$. 
Since $x\in X$ and $r\in (0, \infty)$ are arbitrary,   we conclude  that 
 $\psi\circ d\in \ult(X)$. 

We next prove the implication  $(2)\To (1)$. 
 By the equivalence  of  the conditions $(1)$ and $(3)$ in Theorem \ref{thm:tp}, 
we see that $\psi$ is increasing and amenable. 
In order to show that $\psi$ is  continuous at $0$, 
take an arbitrary strictly decreasing sequence $\{r(n)\}_{n\in \nn}$ with 
$\lim_{n\to \infty}r(n)=0$. Put $r(\infty)=0$, and 
put  $X=\nn\sqcup \{\infty\}$, and 
define an ultrametric  $d$ on $X$ by 
\[
d(n, m)=
\begin{cases}
 r(n)\lor r(m) & \text{if $n\neq m$;}\\
 0 & \text{if $n=m$}. 
\end{cases}
\]
We now consider that $X$ is equipped with the topology induced from $d$. 
Note that $X$ is homeomorphic to the one-point compactification of 
the countable discrete space. 
By the condition (2),  we have  $\psi\circ d\in \ult(X)$. 
Since the point $\infty\in X$ is the unique accumulation point of  $X$,  and since $X$ is compact, 
for every $\epsilon\in (0, \infty)$
the set $X\setminus U(\infty, \epsilon; \psi\circ d)$ is finite. 
Then, by $d(\infty, m)=r(m)$ for all $m\in \nn$, 
we conclude that  for all sufficiently large $n\in \nn$
we have $\psi(r(n))<\epsilon$. This implies that $\psi$ is continuous at $0$. 
\end{proof}

\begin{lem}\label{lem:ultland}
Let $S$ be a range set. 
Let $(X, d)$ be an $S$-valued ultrametric space.
Let $\epsilon\in S_+$. 
Then the function 
$e: X^2\to [0, \infty)$ 
defined by 
$e(x, y)=\min\{d(x, y), \epsilon\}$ 
belongs to $\ult(X, S)$. 
\end{lem}
\begin{proof}
Applying Lemma \ref{lem:ultpreserving} to the map $\psi: [0, \infty)\to [0, \infty)$
 defined by $\psi(x)=\min\{x, \epsilon\}$, 
 we obtain the lemma. 
\end{proof}
Let $(X, d)$ and $(Y, e)$  be metric spaces. 
Define a function 
$d\times_{\infty} e$ 
on 
$(X\times Y)^2$  
by 
\[
(d\times_{\infty} e)((x, y), (z, w))=d(x, z)\lor e(y, w). 
\]
It is well-known that 
$d\times_{\infty} e$ 
is a metric on 
$X\times Y$, 
 and   it generates the product topology of 
 $X\times Y$. 
In the case of ultrametrics,  
we have: 
\begin{lem}
Let $S$ be a range set. 
Let $(X, d)$ and $(Y, e)$  be 
$S$-valued ultrametric spaces. 
Then  the metric 
$d\times_{\infty}e$ 
belongs to 
$ \ult(X\times Y, S)$. 
\end{lem}

For a mutually disjoint family 
$\{T_i\}_{i\in I}$ 
of  topological spaces, 
we consider that  the set 
$\coprod_{i\in I}T_i$ 
is equipped with the direct sum topology.
The following proposition is known as an amalgamation of ultrametrics.  
\begin{prop}\label{prop:ultamal2}
Let $S$ be a range set. 
Let 
$(X, d_X)$ 
and 
$(Y, d_Y)$ be 
$S$-valued   ultrametric spaces. 
If 
$X\cap Y=\emptyset$, 
then for every 
$r\in S_{+}$ 
there exists an $S$-valued  ultrametric 
$h\in \ult(X\sqcup Y, S)$ 
such that 
\begin{enumerate}
	\item $h|_{X^2}=d_X$; 
	\item $h|_{Y^2}=d_Y$; 
	\item for all $x\in X$ and $y\in Y$ we have $r\le h(x, y)$. 
\end{enumerate}
\end{prop}
The proof of Proposition \ref{prop:ultamal2} can be obtained by replacing the symbol ``$+$" with the symbol $``\lor$" 
in the proof of Proposition 3.3 in  \cite{IsI} (or see Theorem 2.2 in \cite{Bog}). 

As a consequence of Proposition \ref{prop:ultamal2}, 
we can construct a one-point extension of an ultrametric space. 

\begin{cor}\label{cor:ultonept}
Let $S$ be a range set possessing at least two elements. 
Let $(X, d)$ be an $S$-valued ultrametric space, and let $o\not\in X$. 
Then there exists an $S$-valued ultrametric 
$D\in \ult(X\sqcup \{o\}, S)$ with 
$D|_{X^2}=d$. 
\end{cor}

\subsection{Invariant metrics on modules}\label{subsec:inv}

Let $R$ be a commutative ring, 
and let 
 $V$ be an $R$-module. 
We say that a metric  
$d$ 
on 
$V$ is \emph{invariant}, 
or 
\emph{invariant under the addition} if 
for all 
$a, x, y\in V$ 
we have 
\[
d(x+a, y+a)=d(x, y).
\] 
By the definitions of  ultra-norms and  invariant metrics, 
we obtain:
\begin{lem}\label{lem:ultinvariant}
Let $R$ be a commutative ring
and let $V$ be an $R$-module with the zero element $0_V$. 
If 
$\|*\|$ is an ultra-normed  on $V$,  
then the metric $d$ on $V$ defined by 
$d(x, y)=\|x-y\|$ 
is an invariant ultrametric on 
$V$. 
Conversely, 
if 
$d$ is  
an invariant ultrametric  on $V$, 
then the function 
$\|*\|: V\to [0, \infty)$ 
defined by 
$\|x\|=d(x, 0_V)$
is an ultra-norm on $V$. 
\end{lem}
Based on Lemma \ref{lem:ultinvariant}, 
in what follows, 
we will identify an ultra-normed $R$-module $(V, \|*\|)$ with  a pair $(V, d)$ of an $R$-module 
and an invariant ultrametric $d$ on $V$. 
When we declare that  $(V, d)$ is an ultra-normed $R$-module,  
the symbol $d$ will be an invariant metric on $V$ constructed in Lemma \ref{lem:ultinvariant}.

By the definition of  ultra-norms, we obtain:
\begin{lem}\label{lem:addconti}
Let $R$ be a commutative ring, 
$(V, d)$ be an ultra-normed   $R$-module. 
Then the addition 
$+:V\times V\to V$ 
and the 
inversion 
$m:V\to V$ 
defined by 
$m(x)=-x$ 
are continuous with respect to the topology induced from 
$d$. 
\end{lem}

The next lemma is utilized in the proof of Theorem \ref{thm:ultemb}. 
\begin{lem}\label{lem:compmod}
Let $R$ be a commutative ring, 
and 
let 
$(V, d)$ 
be an ultra-normed 
$R$-module with the zero element $0_V$.  
If for every non-zero 
$r\in R$ 
and for every 
$v\in V$ 
we have 
$d(r\cdot v, 0_V)=d(v, 0_V)$, 
then there exists an ultra-normed $R$-module 
$(W, D)$
 which contains 
$V$ 
as an 
$R$-submodule 
such
that 
$d=D|_{V^2}$, and the metric space $(W, D)$ is complete, 
and $V$ is a dense subset of $(W, D)$. 
\end{lem}
\begin{proof}
Let $(W, D)$ be the 
completion of the metric space $(V, d)$. 
We introduce an 
$R$-module 
structure into $(W, D)$. 
For all 
$x, y\in W$, 
take sequences 
$\{x_n\}_{n\in \nn}$
 and 
 $\{y_n\}_{n\in \nn}$ 
 in 
 $V$ 
 such that 
 $x_n\to x$ and $y_n\to y$ 
 as 
 $n\to \infty$. 
 Then we define an  addition on $W$ by 
 \[
 x+y=\lim_{n\to \infty}(x_n+y_n). 
 \]
 Since $d$ is an ultra-norm, this addition is well-defined. 
 \par 
 
For every $r\in R$, we define a scalar multiplication on $W$ by 
\[
r\cdot x=\lim_{n\to \infty}r\cdot x_n. 
\]
 By the assumption on the scalar multiplication on 
 $V$ and  the ultra-norm, 
 the scalar multiplication on $W$ 
 is well-defined. 
By these definitions, 
$(W, D)$ 
becomes an ultra-normed 
$R$-module which contains 
$V$ as an $R$-submodule. 
This finishes the proof. 
 \end{proof}

\begin{rmk}
A metric space $(Y, e)$ is a completion of a metric space $(X, d)$
if and only if $(Y, e)$ is complete and $(X, d)$ is isometric to 
a dense subspace of $(Y, e)$. Note that for every metric space 
$(X, d)$ a completion of $(X, d)$ is unique up to isometry. 
Thus, Lemma \ref{lem:compmod} claims 
the existence of some special $R$-module structure on the  completion of the ultra-normed $R$-module 
$(V, d)$. 
\end{rmk}

\subsection{Basic properties of ultrametric spaces}\label{subsec:basic}

The  next lemma follows from the strong triangle inequality. 
\begin{lem}\label{lem:isosceles}
Let 
$X$ 
be a set,  
and let 
$w: X^2\to \rr$ 
be a symmetric map. 
Then 
$w$ 
satisfies the strong triangle inequality if and only if 
for all 
$x, y, z\in X$ 
the inequality 
$w(x, z)<w(y, z)$ 
implies 
$w(y, z)=w(x, y)$. 
\end{lem}

By this lemma, 
we see that  in an ultrametric space, 
every triangle is  isosceles,  
and the side-length of the legs 
of  the isosceles triangle  
is equal to or greater than the side-length  of base. 
\par

The next proposition follows 
from the strong triangle inequality (see (12) in \cite[Theorem 1.6]{CS}). 
\begin{prop}\label{prop:ultcompletion}
Let $S$ be a range set,  
and let 
$(X, d)$ 
be 
an $S$-valued ultrametric space. 
Then the completion of 
$(X, d)$ 
is an 
$S$-valued ultrametric space. 
\end{prop}

\begin{rmk}
It follows from  Proposition \ref{prop:ultcompletion} that, 
for every separable ultrametric space 
$(X, d)$, 
the set 
$\{\, d(x, y)\mid x, y\in X\, \}$ 
is at most countable. 
Moreover, 
using Remark \ref{rmk:ps} we see 
that if $S$ is a range set with $\card(S)\le \aleph_0$,
then there exists a separable ultrametric $(X, d)$ 
such that $S=\{\, d(x, y)\mid x, y\in X\, \}$. 
This phenomenon is a reason why we consider  
$S$-valued ultrametrics. 
Dovgoshey and Shcherbak \cite{DoSh} have recently proven that 
an ultrametrizable topological space $X$ is 
separable if and only if 
for every $d\in \ult(X)$ we have $\card(d(X^2))\le \aleph_0$. 
\end{rmk}

\begin{lem}\label{lem:ultST}
Let 
$S$ 
be a range set with the countable coinitiality. 
Let 
$\{r(i)\}_{i\in \nn}$ 
be a  strictly decreasing sequence in 
$S$ 
such that 
$r(i)\to 0$ as $i\to \infty$. 
Put 
$T=\{0\}\cup\{\, r(i)\mid i\in \nn\, \}$. 
Then, for every topological space $X$, 
from $\ult(X, S)\neq \emptyset$
it follows that 
$\ult(X, T)\neq \emptyset$. 
\end{lem}
\begin{proof}
Take 
$d\in \ult(X, S)$. 
Define a function $\psi:[0, \infty)\to [0, \infty)$ by 
\[
\psi(x)=
\begin{cases}
r(1) &\text{if $r(1)< x$;}\\
r(n) & \text{if $r(n+1)<x\le r(n)$;}\\
0 & \text{if $x=0$. }
\end{cases}
\]
Thus, $\psi$ is increasing, amenable and continuous at $0$. 
Put $e=\psi\circ d$. 
Since $\psi([0, \infty))=T$,  by Lemma \ref{lem:ultpreserving} we have $e\in \ult(X, T)$. 
\end{proof}

\begin{lem}\label{lem:fincoi}
Let $S$ be a range set which does not have  the countable coinitiality. 
Then every $S$-valued ultrametric space 
$(X, d)$ is discrete and complete. 
\end{lem}
\begin{proof}
Since $S$ does not have countable coinitiality, 
there exists  
$r\in (0, \infty)$ such that 
$[0, r)\cap S=\{0\}$. 
Thus for every 
$x\in X$ 
we have 
$U(x, r)=\{x\}$. 
This implies the lemma. 
\end{proof}

We now prove that 
for every range set 
$S$ 
with the countable coinitiality, 
the ultrametrizability 
and the 
$S$-valued ultrametrizability are equivalent to each other.  
\begin{prop}\label{prop:ultequi}
Let 
$S$ 
be a range set. 
The the following statements are equivalent:
\begin{enumerate}
\item The set $S$ has countable coinitiality.
\item For every topological space $X$ we have 
\[
(\ult(X)\neq \emptyset)\iff (\ult(X, S)\neq \emptyset). 
\]

\end{enumerate}
\end{prop}
\begin{proof}
We first prove the implication $(1)\To (2)$. 
Since for every topological space $X$ we have $\ult(X, S)\subseteq \ult(X)$, 
it suffices to show that if 
$\ult(X)\neq \emptyset$, 
then  we have 
$\ult(X, S)\neq \emptyset$. 
Let 
$\{r(i)\}_{i\in \nn}$ 
be a  strictly decreasing sequence in 
$S$ 
such that 
$r(i)\to 0$ as $i\to \infty$. 
Put 
$T=\{0\}\cup\{\, r(i)\mid i\in \nn\, \}$, 
and put 
$A=\{0\}\cup \{\, 2^{-i}\mid i\in \nn\, \}$. 
Then there exists an  increasing amenable   function
$\psi :[0, \infty)\to [0, \infty)$ which is continuous at $0$ and satisfies  $\psi(A)=T$. 
Since 
$\ult(X)\neq \emptyset$, Lemma \ref{lem:ultST} implies 
that $\ult(X, A)\neq \emptyset$. 
Thus, by Lemma \ref{lem:ultpreserving} and $\psi(A)=T$, 
we have 
$\ult(X, T)\neq\emptyset$. 
From 
$\ult(X, T)\subseteq \ult(X, S)$, 
we conclude that $\ult(X, S)\neq \emptyset$. 

We next prove the implication $(2)\To (1)$. 
Suppose that $S$ does not have countable coinitiality. 
Let $X$ be a non-discrete topological space with $\ult(X)\neq \emptyset$ (for example, we can choose $X$ as the Cantor set). 
By Lemma \ref{lem:fincoi},  
we have 
$\ult(X, S)=\emptyset$. 
Thus, we conclude that the implication $(2)\To (1)$ holds true. 
\end{proof}

A topological  space 
$X$ 
is said to be 
\emph{$0$-dimensional} if 
for every pair of  disjoint  closed subsets 
$A$ 
and  
$B$ 
of 
$X$, 
there exists a clopen subset 
$Q$ 
of 
$X$ 
with 
$A\subseteq Q$ 
and 
$Q\cap B=\emptyset$. 
Such a space is sometimes also said to be \emph{ultranormal}.  
Note that a metric space 
$(X, d)$ 
is 
$0$-dimensional 
if and only if 
every finite open covering of 
$X$ 
has 
a  refinement   
consisting of mutually disjoint finite open sets.   
This equivalence   follows from  
the fact that  
for every metric space, 
the large inductive dimension coincides with  the covering dimension
 (see e.g., \cite[Chapter 4, Theorem 5.4]{P}). 

The following was proven by  de Groot \cite{Gr1} (see also \cite{CS}). 
\begin{prop}\label{prop:ult0}
All ultrametrizable spaces are $0$-dimensional. 
\end{prop}

We now clarify 
a relation between the complete 
$S$-valued ultrametrizability 
for a range set $S$
and the complete metrizability. 
The proofs of the following lemma and proposition are 
$S$-valued  ultrametric analogues  of  \cite[Theorem 24.12]{W}. 
\begin{lem}\label{lem:ultopencomp}
Let 
$S$ 
be a range set with the countable coinitiality. 
Let 
$X$ 
be a completely 
$S$-valued ultrametrizable, 
and let 
$G$ 
be an open subset of 
$X$. 
Then 
$G$ 
is completely 
$S$-valued  ultrametrizable. 
\end{lem}
\begin{proof}
Sine $X$ is $0$-dimensional,  
and since all open sets of metric spaces are 
$F_{\sigma}$, 
there exists a sequence $\{O_n\}_{n\in \nn}$ 
of clopen sets of 
$X$ such that
\begin{enumerate}
\item for each $n\in \nn$, we have $O_n\subseteq O_{n+1}$;
\item $G=\bigcup_{n\in \nn}O_n$. 
\end{enumerate} 
Take a sequence of $\{a_n\}_{n\in \nn}$ 
in the field 
$\qq_2$ of  all 
$2$-adic numbers  
such that 
for each 
$m\in\nn$ 
the sum 
$\sum_{i=m}^{\infty}a_i$ 
is convergent to a non-zero 
$2$-adic number 
(for example, we can take $a_n=2^{n}-2^{n+1}$). 
Define a function 
$F: X\to \qq_2$ 
by 
$F(x)=\sum_{i=1}^{\infty}a_i\cdot \chi_{O_i}(x)$, 
where 
$\chi_{O_i}$ 
is the characteristic function of 
$O_i$. 
Then $F$ is continuous. 
By the assumption on 
$\{a_n\}_{n\in \nn}$, 
for every $x\in G$, 
we have $F(x)\neq 0$ 
and $F|_{X\setminus G}= 0$. 
Take a complete 
$S$-valued  ultrametric 
$d\in \ult(X, S)$. 
We denote by 
$v_2: \qq_2\to \zz\sqcup \{\infty\}$ 
the 
$2$-adic  valuation on 
$\qq_2$. 
Take a strictly decreasing sequence 
$\{r(i)\}_{i\in \nn}$ in $S$ 
with 
$\lim_{i\to \infty}r(i)=0$. 
We put $r(\infty)=0$. 
Then a metric 
$W: (\qq_{2})^2\to S$ defined by 
\[
W(x, y)=r(v_2(x-y)\lor 1)
\]
belongs to $\ult(\qq_{2}, S)$, 
and it is complete. 
Define a metric $D$ on $G$ by 
\[
D(x, y)=W\left(\frac{1}{F(x)}, \frac{1}{F(y)}\right)\lor d(x, y). 
\]
Since the function $1/F$ is continuous on $G$, 
we have $D\in \ult(G, S)$. 
We next show that $D$ is complete. 
Assume that  
$\{x_n\}_{n\in \nn}$ is Cauchy in $(G, D)$. 
Then 
$\{1/F(x_n)\}_{n\in \nn}$ and 
$\{x_n\}_{n\in \nn}$  
are Cauchy in 
$(\qq_2, W)$ 
and 
$(X, d)$, 
respectively. 
Thus, 
there exist 
$A\in \qq_2$ 
and 
$B\in X$ 
such that 
 $1/F(x_n)\to A$ 
 in 
 $(\qq_2, W)$, 
 and 
 $x_n\to B$ in $(X, d)$. 
If $B\not\in G$, 
then we have $F(x_n)\to 0$. 
This contradicts to $1/F(x_n)\to A$. 
Thus $B\in G$, and hence $D$ is complete. 
Therefore we conclude that $G$ is completely $S$-valued  ultrametrizable. 
\end{proof}

\begin{prop}\label{prop:completelyult}
Let $S$ be a range set with the countable coinitiality. 
A topological space $X$ is completely $S$-valued ultrametrizable 
if and only if 
$X$ is completely metrizable and $S$-valued ultrametrizable. 
\end{prop}
\begin{proof}
It suffices to show 
that if $X$ is completely metrizable 
and $S$-valued ultrametrizable, 
then $X$ is completely $S$-valued ultrametrizable. 
Take a strictly decreasing sequence 
$\{r(i)\}_{i\in \nn}$ in $S$ with $\lim_{i\to \infty}r(i)=0$. 
Put $T=\{0\}\cup\{\, r(i)\mid i\in \nn\, \}$. 
Then  
$T$ is a range set  with 
$T\subseteq S$. 
Note that $T$ has  countable coinitiality 
and it is a closed set of  $[0, \infty)$. 
By Proposition \ref{prop:ultequi}, 
we can take an ultrametric 
$d\in \ult(X, T)$. 
Let $(Y, D)$ be a completion of $(X, d)$. 
By Proposition \ref{prop:ultcompletion},  
the space $(Y, D)$ is a $T$-valued  ultrametric space. 
Since $X$ is completely metrizable,  
$X$ is $G_{\delta}$ in $Y$ (see \cite[Theorem 4.3.24]{Eng} or \cite[Theorem 24.12]{W}). 
Thus there exists a sequence 
$\{G_n\}_{n\in \nn}$ of open sets in $Y$ such that 
$X=\bigcap_{n\in \nn}G_n$. 
By Lemmas \ref{lem:ultland} and  \ref{lem:ultopencomp}, 
we can take a sequence 
$\{e_n\}_{n\in \nn}$ 
of complete  $T$-valued ultrametrics such that $e_n\in \ult(G_n, T)$ 
and 
$e_n(x, y)\le r(n)$ 
for all 
$x, y\in G_n$ 
and  for all 
$n\in \nn$. 
Define an 
$S$-valued ultrametric 
$h\in \ult(X, S)$ by 
\[
h(x, y)=\sup_{n\in \nn}e_n(x, y). 
\]
Then $h$ is complete.
Since $T$ is a closed set of $[0, \infty)$, 
we have $h\in \ult(X, T)$. 
By $\ult(X, T)\subseteq \ult(X, S)$, 
we obtain a complete $S$-valued ultrametric $h\in \ult(X, S)$. 
%%%%%%%%%%%MMMM
\end{proof}

Combining Lemma \ref{lem:fincoi} and Proposition \ref{prop:completelyult}, 
we obtain: 
\begin{prop}\label{prop:Sultcomp}
Let $S$ be a range set. 
A topological space $X$ is completely $S$-valued ultrametrizable 
if and only if 
$X$ is completely metrizable and $S$-valued ultrametrizable. 
\end{prop}

Since every $G_{\delta}$ subset of every  complete metrizable space is 
completely metrizable,  Proposition \ref{prop:Sultcomp} implies: 

\begin{cor}
Let $S$ be a range set. 
If $X$ is a completely $S$-valued ultrametrizable space,  
then so is every $G_{\delta}$ subset of $X$. 
\end{cor}

\subsection{Continuous functions on $0$-dimensional spaces}\label{subsec:conti}

The following theorem was  stated in  \cite[Theorem 1.1]{D}, and 
a Lipschitz version of it  was proven in \cite[Theorem 2.9]{BDHM}. 
\begin{prop}\label{prop:ultret}
Let $X$ be an ultrametrizable space,  
and let $A$ be a non-empty closed subset of $X$. 
Then there exists a retraction from  $X$ into $A$;  namely, 
 there exists a continuous map $r: X\to A$ with $r|_A=1_A$. 
\end{prop}

\begin{cor}\label{cor:ultes}
Let $X$ be an ultrametrizable space,  
and let $A$ be a closed subset of $X$. 
Let $Y$ be a topological space. 
Then  every continuous map $f: A\to Y$ 
can be extended  to a continuous map from $X$ to $Y$. 
\end{cor}
\begin{proof}
We may assume that $A$ is non-empty. 
By Proposition \ref{prop:ultret}, 
there exists a retraction $r: X\to A$. 
Put $F=f\circ r$, 
then $F$ is a continuous extension of $f$. 
\end{proof}

Let 
$Z$ 
be a metrizable space. 
We denote by 
$\clo(Z)$ 
the set of all non-empty closed  subsets  of $Z$. 
For a topological space $X$ 
we say that a map 
$\phi: X\to \clo(Z)$ 
is 
\emph{lower semi-continuous} 
if for every open subset 
$O$ of $Z$ 
the set 
$\{\, x\in X\mid \phi(x)\cap O\neq \emptyset\, \}$ 
is open in $X$. 

The following  theorem is known as 
the $0$-dimensional  Michael continuous selection theorem, which  was 
first stated in \cite{Mi2}, 
essentially in \cite{Mi1}
(see also \cite[Proposition 1.4]{IMi}). 
\begin{thm}\label{thm:zeroMichael}
Let $X$ be a $0$-dimensional paracompact space, 
and $A$  a closed subsets of $X$. 
Let $Z$ be a completely metrizable space. 
Let 
$\phi:X\to \clo(Z)$ 
be a lower semi-continuous map. 
If  a continuous map 
$f:A\to Z$ 
satisfies 
$f(x)\in \phi(x)$ 
for all 
$x\in A$, 
then there exists a continuous map 
$F:X\to Z$ with $F|_A=f$ 
such that 
for every 
$x\in X$ 
we have 
$F(x)\in \phi(x)$. 
\end{thm}

By the invariance of the ultra-norms under the addition, 
we have:
\begin{prop}\label{prop:ultHiso}
Let $R$ be a commutative ring, 
and let $(V, d)$ be an ultra-normed $R$-module. 
Let
$x, y\in V$. 
Then for every 
$r\in (0, \infty)$ 
we have 
\[
\mathcal{H}(B(x, r), B(y, r))\le d(x, y), 
\]
where 
$\mathcal{H}$ 
is the Hausdorff distance induced from $d$. 
\end{prop}
\begin{proof}
Since  $d$ is invariant  under the addition, 
for every $w\in B(y, r)$
we have 
$x+w-y\in B(x, r)$ 
and 
$d(w, x+w-y)=d(x, y)$. 
Thus, 
$B(y, r)\subseteq B(B(x, r), d(x, y))$. 
Similarly,  
we  obtain 
$B(x, r)\subseteq B(B(y, r), d(x, y))$. 
Therefore, 
$\mathcal{H}(B(x, r), B(y, r))\le d(x, y)$.  
\end{proof}

\begin{cor}\label{cor:ultlsc}
Let $X$ be a topological space,  
Let $R$ be a commutative ring, 
and let $(V, d)$ be an ultra-normed $R$-module. 
Let 
$H:X\to V$ 
be a continuous map and 
$r\in (0, \infty)$. 
Then a map 
$\phi: X\to \clo(V)$ 
defined by 
$\phi(x)=B(H(x), r)$ 
is lower semi-continuous. 
\end{cor}
\begin{proof}
For every open subset $O$ of $V$,  
and for every point 
$a\in X$ 
with 
$\phi(a)\cap O\neq \emptyset$, 
choose  
$u\in \phi(a)\cap O$ 
and  
$l\in (0, \infty)$ 
with 
$U(u, l)\subseteq O$. 
By Proposition \ref{prop:ultHiso} 
and the continuity of $H$, 
we can take 
a neighborhood $N$ of the point $a$ in $X$
such that for every 
 $x\in N$ 
 we have 
 \[
 \mathcal{H}(\phi(x), \phi(a))\le d(H(x), H(a))<l.
 \] 
Then we have 
$\phi(x)\cap U(u, l)\neq \emptyset$, 
and hence  
$\phi(x)\cap O\neq \emptyset$. 
Therefore the set 
$\{\, x\in X\mid \phi(x)\cap O\neq \emptyset\, \}$ 
is open in $X$. 
\end{proof}

The following theorem is known as 
the Stone theorem on the paracompactness, 
which was proven in \cite{St}. 
\begin{thm}\label{thm:Stone}
All metrizable spaces are paracompact. 
\end{thm}
By Theorem \ref{thm:Stone} and Proposition \ref{prop:ult0}, 
we can apply Theorem \ref{thm:zeroMichael} to all ultrametrizable space.

\subsection{Baire spaces}\label{subsec:Baire}
A topological space $X$ is said to be  \emph{Baire} if the intersection of 
every countable family of  dense open subsets of $X$ is  dense in $X$. 

\par

The following  is  known as the Baire category theorem. 
\begin{thm}
Every completely metrizable space is a Baire space. 
\end{thm}

Since 
$G_{\delta}$ 
subset of completely metrizable space is 
completely metrizable (see, e.g.  \cite[Theorem 24.12]{W}), 
we obtain the following:
\begin{lem}\label{lem:gdelta}
Every 
$G_{\delta}$ 
subset of a completely metrizable space is a Baire space. 
\end{lem}

%%%%%%%%%%%%%%%%%%%%%%%%%%%%%%%%%%
%%%%%%%%%%%%%%%%%%%%%%%%%%%%%%%%%%%%%%%%%%%%%%%%%%%%%%%%%%%%%%%%%%%%%%%%%%%%%%%%%%%%%%%%%%%%%%%%%%%%%%%%%%%%%%%%%%%%%%%%%%%%%%%%%%%%%%%%%%%%%%%%%%%%%%%%%%%%%%%%%%%%%%%%%%%%%%%%%%%%%%%%%%%%%%%%%%%%%%%%%%%%%%%%%%%%%%%%%%%%%%%%%%%%%%%%%%

\section {An embedding theorem of ultrametric spaces }\label{sec:uae}
In this section, we prove Theorem \ref{thm:ultemb}. 
\subsection{Proof of Theorem \ref{thm:ultemb}}\label{subsec:31}
We first discuss general algebraic facts. 
\begin{lem}\label{lem:algind}
Let $R$ be a commutative ring, 
and let $V$ be an $R$-module. 
Let $P$ be an $R$-independent set of $V$, 
and let 
$Q$ be a subset of $P$.
Let $H$ be an $R$-submodule of $V$ 
generated by $Q$. 
Then $H\cap P=Q$.  
\end{lem}
\begin{proof}
By the definition of $H$, 
first we have $Q\subseteq H\cap P$. 
Since $P$ is $R$-independent, 
we have $(P\setminus Q)\cap H=\emptyset$. 
Thus every 
$x\in H\cap P$ 
must belong to $Q$, 
and hence 
we conclude that $H\cap P\subseteq Q$. 
\end{proof}

Let $R$ be a commutative ring. 
Let $X$ be a set,  
and let 
$o\not\in X$.
We denote by 
$\mathrm{F}(R, X, o)$ 
the free $R$-module $M$ 
satisfying  that 
\begin{enumerate}
\item $X\sqcup\{o\}\subseteq M$;
\item $o$ is the zero element of $M$;
\item $X$ is an $R$-independent generator of $M$. 
\end{enumerate}
Note that by the construction of free modules, 
$\mathrm{F}(R, X, o)$ uniquely exists up to isomorphism. 

\par
For two sets 
$A$, 
$B$, 
we denote by 
$\map(A, B)$ 
the set of all maps from $A$ into $B$. 
Let $R$ be a commutative ring, 
and let $V$ be an $R$-module. 
Let $E$ be a non-empty set. 
Then the set $\map(E, V)$ becomes an $R$-module with  the
coordinate-wise addition and scalar multiplication. 
Note that the zero element of $\map(E, V)$ is 
the zero function of $\map(E, V)$; namely the constant function valued at
 the zero element of $V$. 
In what follows, 
the set $\map(E, V)$ will be always equipped with this module structure. 
\par

We  next discuss 
a construction of universal ultrametric spaces of Lemin--Lemin type \cite{LL}. 
Let $S$ be a range set. 
Let $M$ be a set,  
and let 
$o\in M$ 
be a fixed base point. 
A map 
$f:S_{+}\to M$ 
is said to be 
\emph{eventually $o$-valued} if 
there exists 
$C\in S_+$ 
such that 
for every 
$q>C$ 
we have 
$f(q)=o$. 
We denote by  
$\elel(S, M, o)$ 
the set of all eventually $o$-valued maps from $S_+$ to $M$.
Define a metric $\Delta$ on $\elel(S, M, o)$ by 
\[
\Delta(f, g)=\sup\{\, q\in S_+\mid f(q)\neq g(q)\, \}. 
\] 
Note that 
$\Delta$ takes values in the closure $\cl(S)$ of $S$ in $[0, \infty)$. 
\par

The next lemma follows from the definitions of 
$\elel(S, M, o)$ and $\Delta$. 
\begin{lem}\label{lem:LLcomp}
For every range set $S$ possessing at least two elements, 
for every set $M$ and for every point $o\in M$, 
the space $(\elel(S, M, o), \Delta)$ is a complete 
$\cl(S)$-valued  ultrametric space. 
\end{lem}
In the next theorem, 
we review  the Lemin--Lemin construction \cite{LL} of  
embeddings  into their universal spaces in order   to 
obtain  more detailed  information  of their construction. 
\begin{thm}\label{thm:LL0}
Let $S$ be a range set possessing at least two elements. 
Let $(X\sqcup \{o\}, d)$ be an 
$S$-valued  ultrametric space with 
$o\not\in X$. 
Let $K$ be a set with 
$X\sqcup\{o\} \subseteq K$.  
Then there exists an isometric embedding 
$L: X\sqcup\{o\}\to \elel(S, K, o)$ such that 
%%%%%%%%%%%%%%%MMM
\begin{enumerate}
\item for every $q\in S_+$ we have $L(o)(q)=o$;
\item for every $x\in X$ the function  $L(x)$ is valued in $X\sqcup\{o\}$; 
\item for all $x, y\in X$ we have 
\[
(0, d(x, y)]\cap S_{+}=\{\, q\in S_+\mid L(x)(q)\neq L(y)(q)\, \}. 
\]
\end{enumerate}
\end{thm}
\begin{proof}
Let 
$X\sqcup\{o\}=\{x(\alpha)\}_{\alpha<\kappa}$ 
be an injective index with $x(0)=o$, 
where $\kappa$ is a cardinal. 
By following the Lemin--Lemin's way \cite{LL}, 
we construct an isometric embedding 
$L: X\to \elel(S, K, o)$ by transfinite induction. 
Put $L(x(0))= o$. 
Let $\gamma <\kappa$. 
Assume that 
an isometric embedding 
$L: \{\, x(\alpha)\mid \alpha<\gamma\, \}\to \elel(S, K,  o)$ is already defined. 
Set
$D_{\gamma}=\inf\{\, d(x(\alpha), x(\gamma))\mid \alpha<\gamma\, \}$. 
\par

Case 1. 
(There exists an ordinal
$\beta<\gamma$ 
with 
$D_{\gamma}=d(x(\beta), x(\gamma))$. ~) 
We define an eventually $o$-valued map 
$L(x(\gamma)): S_{+}\to K$ 
by 
\[
L(x(\gamma))(q)=
\begin{cases}
x(\gamma) & \text{if $q\in (0, D_{\gamma}]$; }\\
L(x(\beta))(q) & \text{ if $q\in (D_{\gamma}, \infty)$. }
\end{cases}
\]
\par

Case 2. 
(No ordinal 
$\beta<\gamma$ 
satisfies 
$D_{\gamma}=d(x(\beta), x(\gamma))$. ~)
Take a sequence 
$\{\alpha_n\}_{n\in \nn}$ 
with 
$\alpha_n<\gamma$ and 
$d(x(\alpha_n), x(\gamma))<D_{\gamma}+1/n$ 
for all 
$n\in \nn$. 
We define an eventually $o$-valued map 
$L(x(\gamma)): S_{+}\to K$ by
\[
L(x(\gamma))(q)=
\begin{cases}
x(\gamma) & \text{if $q\in (0, D_{\gamma}]$; }\\
L(x(\alpha_n))(q) & \text{if $D_{\gamma}+1/n<q$.}
\end{cases}
\]
Similarly to \cite{LL}, 
the map 
$L: X\sqcup \{o\}\to \elel(S, K, o)$ is  well-defined and isometric, 
and 
we see that the conditions (1) and (2) are satisfied.  
\par

We now prove that the condition (3) is satisfied. 
Note that for each 
$\alpha<\kappa$, 
the function 
$L(x(\alpha))$ is valued in 
$\{\, x(\beta)\mid\beta\le \alpha\, \}$. 
Let $\gamma<\kappa$.  
Assume that for all 
$\alpha, \beta<\gamma$, 
the condition (3) is satisfied for 
$x=x(\alpha)$ 
and 
$y=x(\beta)$. 
We prove that for every $\alpha<\gamma$, 
the condition (3) is satisfied 
for $x=x(\alpha)$ and $y=x(\gamma)$. 
\par

In Case 1, 
by the definition of 
$D_{\gamma}$ 
and 
$\beta<\gamma$, 
we have 
$d(x(\beta), x(\gamma))\le d(x(\alpha), x(\gamma))$. 
From this inequality and the strong triangle inequality 
(or Lemma \ref{lem:isosceles}), 
it follows that
$d(x(\alpha), x(\beta))\le d(x(\alpha), x(\gamma))$. 
Thus, 
by the hypothesis of transfinite induction 
and the definition of 
$L(x(\gamma))$, 
we conclude that the condition (3) is satisfied. 

In Case 2, 
by the definition of 
$D_{\gamma}$, 
we have 
$D_{\gamma}<d(x(\alpha), x(\gamma))$, 
and 
for all sufficiently large 
$n\in \nn$,  
we obtain 
$d(x(\alpha_n), x(\gamma))<d(x(\alpha), x(\gamma))$.  
 Lemma \ref{lem:isosceles} implies that 
 $d(x(\alpha_n), x(\alpha))=d(x(\alpha), x(\gamma))$. 
Since on the set 
$(D_{\gamma}+1/n, d(x(\alpha), x(\alpha_n))]$ 
the function 
$L(x(\gamma))$ 
coincides with 
$L(x(\alpha_n))$, 
by the hypothesis of transfinite induction
we have 
\[
(D_{\gamma}+1/n, d(x(\alpha), x(\gamma))]\cap S_+
\subseteq 
\{\, q\in S_+\mid L(x(\alpha))(q)\neq L(x(\gamma))(q)\, \}. 
\]
By 
$L(x(\alpha))(S_+)\subseteq \{\, x(\beta)\mid\beta\le \alpha\, \}$ 
and $L(x(\gamma))|_{(0, D_{\gamma}]}= x(\gamma)$, 
we also have 
\[
(0, D_{\gamma}]\cap S_+
\subseteq
 \{\, q\in S_+\mid L(x(\alpha))(q)\neq L(x(\gamma))(q)\, \}. 
\]
These imply the condition (3) for $x=x(\alpha)$ and $y=x(\gamma)$. 
\end{proof}

\begin{rmk}
An ultrametric space $X$ is said to be  
\emph{universal} for a class of ultrametric space if 
every ultrametric space in the class is isometrically embeddable into $X$. 
In \cite{LL}, for each cardinal $\tau$,  
Lemin and Lemin constructed 
a universal ultrametric space for the class of all ultrametric spaces of topological weight $\tau$. 
There are several studies on universal ultrametric spaces. 
Vaughan \cite{Vau99} studied 
 universal ultrametric spaces of Lemin--Lemin-type. 
Vestfrid \cite{V94}, Gao and Shao \cite{GS11},  
and Wan \cite{Wan} studied 
  universal ultrametric spaces of Urysohn-type. 
\end{rmk}

The following lemma plays  
a central role in the  construction  of our   embeddings
from  ultrametric spaces into ultra-normed modules. 
\begin{lem}\label{lem:ultcood}
Let $S$ be a range set possessing at least two elements. 
Let $R$ be a commutative ring. 
Let $M$ be  an $R$-module. 
Let $0_R$ and  $0_M$ denote the zero elements of $R$ and  $M$, respectively. 
Then the following statements hold true: 
\begin{enumerate}
\item The space $\elel(S, M, 0_M)$ becomes an $R$-submodule of 
$\map(S_+, M)$. 
\item The ultrametric $\Delta$ on $\elel(S, M, 0_M)$ is invariant under the addition;
namely, $(\elel(S, M, 0_M), \Delta)$ is  ultra-normed.
\item If $R$ is an integral domain, %%%%%%%%%%%%%%MM
and  if $M$ is torsion-free, then for every  $r\in R\setminus\{0_R\}$, and 
for every  $x\in \elel(S, M, 0_M)$, 
the equality  $\Delta(r\cdot x, 0_{\elel})=\Delta(x, 0_{\elel})$ holds true, 
where $0_{\elel} \in \elel(S, M, 0_M)$ is the constant map valued at $0_M$. 
\end{enumerate}
\end{lem}
\begin{proof}
The statement (1) follows from 
$\elel(S, M, 0_M)\subseteq \map(S_+, M)$ 
and the definition of  eventually $0_M$-valued maps. 
We prove the statement  (2). 
For all 
$f, g, h\in \elel(S, M, 0_M)$, 
and for every $q\in S_+$, 
we have 
$f(q)\neq g(q)$ 
if and only if  
$f(q)+h(q)\neq g(q)+h(q)$. 
Thus, by the definition,  
$\Delta$ is invariant under the addition. 
By the similar way, 
since $R$ is an integral domain, 
we see that the statement (3) holds true. 
\end{proof}

\begin{lem}\label{lem:ultindependent}
Let $R$ be a commutative ring. 
Let $S$ be a range set possessing at least two elements. 
Let $(X\sqcup\{o\}, D)$ be an 
$S$-valued  ultrametric space with  
$o\not\in X$. 
Put $M=\mathrm{F}(R, X, o)$.  
Let 
$L: X\sqcup \{o\}\to \elel(S, M, o)$ 
be an isometric embedding 
constructed in Theorem \ref{thm:LL0}. 
Then  
 $L(X)$ 
 is 
 $R$-independent in the 
 $R$-module 
 $\elel(S, M, o)$. 
\end{lem}
\begin{proof}
In this proof,  
we denote by 
$0_M$ 
the zero element $o$ of $M$.
Let $C=\{ x_1, \dots, x_{n} \}$ be an arbitrary finite subset of $X$. 
Assume that
\[
\sum_{i=1}^{n}N_i\cdot L(x_i)=0_{\elel}, 
\]
where 
$N_i\in R$ 
for all $i\in \{1, \dots, n\}$ 
and 
$0_{\elel}$ 
stands for  the zero function of 
$\elel(S, M. o)$. 
Put 
\begin{align*}
c&=
\min\{\,\Delta(L(x), L(y))\mid \text{$x, y\in C\sqcup\{o\}$ and $x\neq y$\,}\}. 
\end{align*}
Since $L$ is isometric, we have 
$c\in S_+$. 
By the definition of 
$\Delta$ 
and the conditions (1) and (3) stated in  Theorem \ref{thm:LL0}, 
we see that for all 
$i, j\in \{1, \dots, n\}$ 
we have  $L(x_i)(c)\neq L(x_j)(c)$,  
and for each $i\in \{1, \dots, n\}$  we have 
$L(x_i)(c)\neq 0_M$. 
By the definition of 
$\mathrm{F}(R, X, o)$, 
we see that  the set 
$\{ L(x_1)(c), \dots, L(x_n)(c)\}$ 
is 
$R$-independent in $M$. 
Since 
\[
\sum_{i=1}^{n}N_i\cdot L(x_i)(c)=0_M, 
\]
we have $N_i=0_R$ for all $i\in \{1, \dots, n\}$, 
where $0_R$ is the zero element of $R$. 
Thus 
$\{ L(x_1), \dots, L(x_n) \}$ 
is 
$R$-independent in $\elel(S, M, o)$. 
Since the set  $C=\{ x_1, \dots, x_{n} \}$ 
is arbitrary,  
we see that 
$L(X)$ is 
$R$-independent in the module 
$\elel(S, M, o)$. 
\end{proof}

\begin{lem}\label{lem:ultgenmod}
Let $R$ be a commutative ring. 
Let $S$ be a range set possessing at least two elements. 
Let 
$(X\sqcup\{o\}, D)$ 
be an 
$S$-valued  ultrametric space with  
$o\not\in X$. 
Put 
$M=\mathrm{F}(R, X, o)$.  
Let 
$L: X\sqcup \{o\}\to \elel(S, M, o)$ 
be an isometric embedding 
constructed in Theorem \ref{thm:LL0}. 
Let $Q$ be an $R$-submodule of 
$\elel(S, M, o)$ 
generated by 
$L(X)$. 
Then  
the metric 
$\Delta|_{Q^2}$ 
takes values in the range set  $S$. 
\end{lem}
\begin{proof}
In this proof,  we denote by 
$0_M$ the zero element $o$ of $M$. 
Let $0_R$ denote the zero element of $R$. 
\par

By the invariance of $\Delta$ under the addition,  
it suffices to show that for every 
$x\in Q$ we have $\Delta(x, 0_{\elel})\in S$, 
where $0_{\elel}$ is the zero function in
$\elel(S, M, o)$. 
Take 
$x\in Q$. 
Then there exist a finite subset 
$\{ x_1, \dots, x_n \}$ 
of $X$ and 
a finite subset 
$\{N_1, \dots, N_n\}$ of 
$ R\setminus\{0_R\}$ 
such that 
$x=\sum_{i=1}^nN_iL(x_i)$. 
Let 
$p_0, p_1, \dots, p_k$ 
be a sequence in $S$ such that  
\begin{enumerate}
\item $p_0=0$; 
\item $p_j<p_{j+1}$ for all $j$; 
\item $\{\, d(x_i, 0_M)\mid i=1, \dots, n\, \}\cup \{\, d(x_i, x_j)\mid i\neq j\, \}
=\{ p_1, \dots, p_k\}$. 
\end{enumerate}
For 
$j\in \{ 0, \dots, k-1 \}$, 
we put 
$I(j)=(p_j, p_{j+1}]\cap S$, 
 and we put  
 $I(k)=(p_k, \infty)\cap S$. 
By the definition of 
$\{p_j\}_{j=0}^{k}$, 
and by the properties (2) and  (3) of the map 
$L$ stated in Theorem \ref{thm:LL0},  we obtain:
\begin{enumerate}
\renewcommand{\labelenumi}{(\Alph{enumi})}
\item for all 
$a\in \{1, \dots, n \}$ 
we have 
$L(x_a)=0_M$ 
on $I(k)$;
\item for every  
$a\in \{1, \dots, n \}$, 
and  for every 
$j\in \{ 0, \dots, k\}$, 
if there exists  $c\in I(j)$ with 
$L(x_a)(c)=0_M$, 
then  we have 
$L(x_a)=0_{M}$ on $I(j)$; 
\item for all 
$a, b\in \{1, \dots, n \}$, 
and for every 
$j\in \{ 0, \dots, k \}$, 
if there exists  
$c\in I(j)$ 
with 
$L(x_a)(c)= L(x_b)(c)$, 
then we have 
$L(x_a)=L(x_b)$ on $I(j)$. 
\end{enumerate}
Suppose that 
$\Delta(x, 0_{\elel})\not \in S$. 
By the property (A), 
we can take 
$j\in \{ 0, \dots, k-1\}$ 
such that 
$\Delta(x, 0_{\elel})\in I(j)$. 
By the definition of $\Delta$, 
there exists  $p\in I(j)$ with 
$x(p)\neq 0_M$, and  we see that 
$x(p_{j+1})=0_M$. 
Put $q=p_{j+1}$. 
Take a subset 
$\{y_1, \dots, y_m \}$ of 
$\{ x_1, \dots, x_n \}$ such that 
\begin{enumerate}
\renewcommand{\labelenumi}{(\alph{enumi})}
\item\label{item:a} $L(y_1)(q), \dots, L(y_m)(q)$ 
are not equal to the zero element 
$0_M$ of $M$, 
and they are different to each other;
\item $m$ is   maximal  in cardinals  of all subsets of  
the set 
$\{x_1, \dots, x_n\}$ 
satisfying the property (a). 
\end{enumerate}
By the properties (B) and (C), 
the set 
$\{ L(y_1)|_{I(j)}, \dots, L(y_m)|_{I(j)}\}$
 is a maximal $R$-independent subset  of 
 $\{ L(x_1)|_{I(j)}, \dots, L(x_n)|_{I(j)}\}$
  in the $R$-module 
$\map(I(j), M)$. 
Then there exists a subset 
$\{ C_1, \dots, C_{m} \}$ 
of 
$R$ such that 
\[
x|_{I(j)}=\sum_{l=1}^{m}C_{l}L(y_l)|_{I(j)}. 
\]
Since $x(q)=0_M$, 
we have 
\begin{align*}\label{al:ind}
\sum_{l=1}^{m}C_{l}L(y_l)(q)=0_M. 
\end{align*}
Since  
$\{ L(y_1)(q), \dots, L(y_m)(q)\}$ is a subset of $X$, 
it is  
$R$-independent in $M$. 
Then we have $C_l=0_R$ for all 
$l\in  \{ 1, \dots, m \}$, 
and hence
$x=0_M$ on $I(j)$.
This contradicts the existence of 
$p\in I(j)$ with $x(p)\neq 0_M$.  
Therefore, 
$\Delta(x, 0_{\elel})\in S$. 
This completes the proof. 
\end{proof}

Before proving  Theorem \ref{thm:ultemb}, 
we recall that 
a  free module on an integral domain is always  torsion-free. 
\begin{proof}[Proof of Theorem \ref{thm:ultemb}]
Let $S$ be a range set possessing at least two elements. 
Let $R$ be an integral domain, and let  
$(X, d)$ be an ultrametric space. 
\par

We first deal with the case where $(X, d)$ is complete. 
Take  $o\not \in X$. 
Put $M=\mathrm{F}(R, X, o)$. 
Let $(X\sqcup\{o\}, D)$ 
be a one-point extension of $(X, d)$ 
(see Corollary \ref{cor:ultonept}). 
Let 
$L: (X\sqcup\{o\}, D)\to (\elel(S, M, o), \Delta)$ 
be an isometric embedding stated in 
Theorem \ref{thm:LL0}. 
Let 
$Q$ 
be an $R$-submodule of 
$\elel(S, M, o)$ generated by 
$L(X)$, 
and let 
$(V, \Xi)$ be the completion of 
$(Q, \Delta|_{Q^2})$. 
By Lemmas \ref{lem:compmod} and \ref{lem:ultgenmod}, 
and Proposition \ref{prop:ultcompletion}, 
the space 
$(V, \Xi)$ 
is an $S$-valued ultra-normed $R$-module. 
Since complete metric  subspaces are closed in metric spaces, 
  Lemma \ref{lem:ultindependent} implies that 
   $(V, \Xi)$ and $L: (X, d)\to (V, \Xi)$ satisfy the conditions (1) and (2)
    stated in Theorem \ref{thm:ultemb}. 
 Moreover, the latter part of the theorem is also proven. 
\par

In the case where  
$(X, d)$ is not complete, 
let $(Y, e)$ be the completion of $(X, d)$. 
As in the above, we can take 
an ultra-normed $R$-module 
$(W, D)$ and 
an isometric embedding 
$I: Y\to W$  satisfying 
the conditions (1) and (2) in Theorem \ref{thm:ultemb}. 
Let $H$ be an $R$-submodule  of $W$ generated  by $I(X)$. 
Since $I(Y)$ is $R$-independent, 
 Lemma \ref{lem:algind} yields 
 $H\cap I(Y)=I(X)$. 
Thus $I(X)$ is closed in $H$,  
and hence 
$(H, D|_{H^2})$ and $I$ are  desired ones.  
This completes the proof of Theorem \ref{thm:ultemb}. 
\end{proof}

\begin{rmk}
If a range set $S$ is  closed under the supremum operator, 
then 
we can prove Theorem \ref{thm:ultemb} for 
not only an integral domain but also a commutative ring $R$. 
In this case, 
the space 
$(\elel(S, M, o), \Delta)$ 
is an $S$-valued ultrametric space, 
and 
 in the proof of Theorem \ref{thm:ultemb}, 
 we can use the space 
 $(\elel(S, M, o), \Delta)$ 
 instead of the space 
 $(V, \Xi)$. 
\end{rmk}

\subsection{Ultrametrics taking values in general totally ordered sets}\label{subsec:genult}
We say that an ordered set is 
\emph{bottomed} if it  has a least element. 
Let $(T, \le_T)$ be a bottomed totally ordered set. 
Let $X$ be a set. 
A function 
$d: X\times X\to T$  is said to be a 
\emph{$(T, \le_T)$-valued ultrametric on $X$} 
if the following are satisfied:
\begin{enumerate}
\item for all 
$x, y\in X$ 
we have 
$d(x, y)=0_T$ 
if and only if 
$x=y$, 
where 
$0_T$ stands for the least element of 
$(T, \le_T)$;
\item for all 
$x, y\in X$ 
we have 
$d(x, y)=d(y, x)$;
\item for all 
$x, y, z\in X$ 
we have 
$d(x, y)\le_T d(x, z)\lor_T d(z, y)$, 
where 
$\lor_T$ 
is the maximal operator of 
$(T, \le_T)$. 
\end{enumerate}
Such general ultrametric spaces, 
or general metric spaces 
on which  distances are valued 
in  a totally ordered Abelian group are studied for a long time 
(see 
e.g., 
\cite{Sik50}, \cite{CG50}, \cite{SV96},  \cite{OS99} and  \cite{C19}). 
\par

The construction of universal ultrametric space of Lemin--Lemin- type
 mentioned above 
 and 
the proof of Theorem \ref{thm:ultemb} are still valid 
for $(T, \le_T)$-valued ultrametric spaces 
for all bottomed totally ordered set $(T, \le_T)$. 
For  simplicity,  
and for necessity of our study, 
we omit the details of the proof of the following: 
\begin{thm}
Let $(T, \le_T)$ be a bottomed totally ordered set possessing at least two elements. 
Let $R$ be an integral domain, and let  
$(X, d)$ be a $(T, \le_T)$-valued ultrametric space. 
Then there exist a 
$(T, \le_T)$-valued ultra-normed $R$-module  $(V, \|*\|)$,   
and an isometric  embedding $I: X\to V$
such that 
\begin{enumerate}
\item $I(X)$ is closed in $V$;
\item $I(X)$ is $R$-independent in $V$. 
\end{enumerate}
Moreover, 
if $(X, d)$ is complete, 
then we can choose $(V, \|*\|)$ as a complete  
$(T, \le_T)$-valued ultrametric space.
\end{thm}
For a bottomed totally ordered set $(T, \le_T)$, 
we define the 
\emph{coinitiality} 
$\coi(T, \le_T)$ of $T$ as the minimal cardinal $\kappa>0$
such that there exists a strictly decreasing map 
$f: \kappa+1\to T$ with 
$f(\kappa)=0_T$ such that 
for every $t\in T$, there exists $\alpha<\kappa$ with $f(\alpha)\le t$. 
Note that a range set $S$ has  countable coinitiality 
if and only if 
$\coi(\cl(S), \le)=\omega_0$. 
Some  readers may think our results 
such as 
Corollary \ref{cor:ultcpt} 
and Theorems \ref{thm:ultex}--\ref{thm:ultint} 
in this paper can be generalized 
for $(T, \le_T)$-valued  ultrametrics 
for a bottomed totally ordered set   
$(T, \le_T)$ 
satisfying $\coi(T, \le_T)>\omega_0$. 
Unfortunately, 
it seems to be quite difficult. 
Our  proofs of Theorems \ref{thm:ultex}--\ref{thm:ultint} 
require 
the extension theorem 
(Corollary \ref{cor:ultes}) 
of continuous functions on ultrametric spaces. 
An analogue for  $(T, \le_T)$-valued ultrametric spaces 
of  Corollary \ref{cor:ultes} seems not to hold true.

%%%%%%%%%%%%%%%%%%%%%%%%%%%%%%%%%%%%%%%%%%%%%%%%%%%%%%%%%%%%%%%%%%%%%%%%%%%%%%%%%%%%%%%%%%%%%%%%%%%%%%%%%%%%%%%%%%%%%%%%%%%%%%%%%%%%%%%%%%%%%%%%%%%%%%%%%%%%%%%%%%%%%%%%%%%%%%%%%%%%%%%%%%%%%%%%%%%%%%%%%%%%%%%%%%%%%%%%%%%%%%%%%%%%%%%%%%%%%%%%%%%%%%%%%%%%%%%%%%%%%%%%%%%%%%%%%%%%%%%%%%%%%%%%%%%%%%%%%%%%%%%%%%%%%%%%%%%%%%%%%%%%%%%%%%%%%%%%%%%%%%%%%%%%%%%
\section{An extension theorem of ultrametrics}\label{sec:uh}
In this section, 
by following the methods of Toru\'nczyk \cite{Tor} 
and Hausdorff \cite{Haus1},  
we prove Theorem \ref{thm:ultex} and Corollary \ref{cor:ultcpt}. 
Since  Toru\'nczyk's proof of Lemma in \cite{Tor} 
on real linear spaces does not depend on the coefficient ring $\rr$, 
we can apply that method 
to all ultra-normed modules over all commutative rings. 
Toru\'nczyk used 
the Dugundji extension theorem in the proof of Lemma in  \cite{Tor}. 
Instead of the Dugundji extension theorem, 
we use Corollary \ref{cor:ultes}, 
which is an extension theorem 
for continuous functions on ultrametrizable spaces. 
\begin{lem}\label{lem:tor}
Let $R$ be a commutative ring. 
Let $(E, D_E)$ and $(F, D_F)$ be two ultra-normed $R$-modules. 
Let $K$ and $L$ be  closed subsets of $E$ and $F$, respectively. 
Let $f: K\to L$ be a homeomorphism. 
Let $0_E $ and $0_F$ denote the zero elements of $E$ and $F$, respectively. 
Let $g: K\times \{0_F\}\to \{0_E\}\times L$ be 
a homeomorphism defined by 
$g(x, 0_F)=(0_E, f(x))$. 
Then there exists a homeomorphism 
$h:E\times F\to E\times F$ with $h|_{K\times \{0_F\}}=g$. 
\end{lem}
\begin{proof}
By Corollary \ref{cor:ultes}, 
we obtain a continuous  map 
$\beta: F\to E$ 
which is an extension of 
$f^{-1}: L\to K$. 
Define a map 
$J:E\times F\to E\times F$ 
by 
$J(x, y)=(x+\beta(y),y)$. 
Lemma \ref{lem:addconti} implies that 
the addition and the inversion on $E$ is continuous, 
and hence $J$ is continuous. 
The map 
$Q:E\times F\to E\times F$ 
defined by 
$Q(x, y)=(x-\beta(y),y)$ 
is also continuous,  
and it is  the inverse map of 
$J$, 
and hence 
$J$ 
is a homeomorphism. 
Similarly,  
by Corollary \ref{cor:ultes}, 
we obtain a continuous map 
$\alpha: E\to F$ 
which is an extension of 
$f: K\to L$. 
Define a map 
$I:E\times F\to E\times F$ 
by 
$I(x, y)=(x, y+\alpha(x))$. 
Then  
$I$ 
is a homeomorphism. 
Define a homeomorphism 
$h:E\times F\to E\times F$ 
by 
$h=J^{-1}\circ I$.   
Since for every 
$x\in K$ 
we have
$I(x,0_L)=(x,\alpha(x))=(x,f(x))$, 
we obtain
\begin{align*}
h(x,0_F)&=J^{-1}(x, f(x))=Q(x, f(x))=(x-\beta(f(x)),f(x))\\
&=(x-f^{-1}(f(x)),f(x))
=(0_E,f(x))=g(x,0_F), 
\end{align*}
and hence 
$h$ 
is an extension of 
$g$. 
\end{proof}

\begin{proof}[Proof of Theorem \ref{thm:ultex}]
Let $S$ be a range set. 
Let $X$ 
be an $S$-valued  ultrametrizable space, 
and let $A$ be a closed subset of $X$. 
Let  
$e\in \ult(A, S)$. 
Take  
$d\in \ult(X, S)$. 
Theorem \ref{thm:ultemb} implies that
there exist an 
$S$-valued  ultra-normed 
$\zz$-module 
$(E, D_E)$ and 
 a closed  isometric embedding 
 $i: (X, d)\to (E, D_E)$. 
 Similarly, 
 there exist an 
 $S$-valued  ultra-normed 
 $\zz$-module 
 $(F, D_F)$ 
 and 
 a closed isometric embedding 
 $j: (A,  e)\to (F, D_F)$.  
 Let $0_E $ and $0_F$ denote the zero elements of $E$ and $F$, respectively. 
 \par
 
Since 
$A$ is closed in $X$, 
the set 
$i(A)$ is closed in $E$. 
Since 
$i$ 
and 
$j$ are topological embeddings, 
$i(A)$ and $j(A)$
are homeomorphisms. 
Define a map  
$f: i(A)\to j(A)$ 
by 
$f=j\circ (i|_A)^{-1}$,  
and by applying  Lemma \ref{lem:tor} to $f$, 
we obtain a homeomorphism 
$h:E\times F\to E\times F$ 
which is an extension of the map
$g:i(A)\times \{0_F\}\to \{0_E\}\times j(A)$ 
defined by 
$g(i(a),0_F)=(0_E, j(a))$. 
\par

Let 
$k:E\to E\times F$
%%%%%%%%%%%%%MMMMM 
be a natural embedding defined by 
$k(x)=(x,0_F)$. 
The map 
$H:X\to E\times F$ 
defined by 
$H=h\circ k\circ i$
is a topological embedding. 
Define a metric $D$ on $X$ by 
\[
D(x,y)=(D_E\times_{\infty} D_F)(H(x), H(y)).
\]
Then 
$D\in \ult(X, S)$.  
Since for every 
$a\in A$ 
we have 
$H(a)=(0_E, j(a))$, 
and 
since 
$j:(A,\rho)\to (F, D_F)$ 
is an isometric embedding, 
we have 
$D|_{A^2}=e$. 
This completes the proof of the former part.
\par

We next show the latter part. 
Assume that  
$X$ is completely metrizable,  
and  
$e\in \ult(A, S)$ is  complete. 
Then by Proposition \ref{prop:Sultcomp}, 
we can choose 
$d\in \ult(X, S)$ as a complete 
$S$-valued ultrametric.  
Thus, 
we can choose 
$(E, D_E)$ and 
$(F, D_F)$ as complete ultrametric spaces, 
and hence the metric space 
$(X, D)$ can be regarded as 
a closed metric subspace of the complete metric space 
$(E\times F, D_F\times_{\infty}D_E)$. 
Therefore $D$ is complete. 
This finishes the proof. 
\end{proof}
\begin{rmk}
In the proof of Theorem \ref{thm:ultex}, 
for simplicity, 
we use $\zz$-modules. 
The proof described  above  is still valid even 
if we use any integral domain as a coefficient ring. 
\end{rmk}

We next prove Corollary \ref{cor:ultcpt}, 
which characterizes the compactness 
in  terms of the completeness  of ultrametrics. 
\begin{lem}\label{lem:ultcb}
Let $S$ be a range set with the countable coinitiality. 
Let $M$ be a countable discrete space. 
Then there exists a non-complete 
$S$-valued ultrametric 
$d\in \ult(M, S)$. 
\end{lem}
\begin{proof}
Take a strictly decreasing sequence  
$\{a(i)\}_{i\in \nn}$ in  $S$  convergent to $0$ as $i\to \infty$. 
We may assume that $M=\nn$. 
Define a metric $d$ on $M$ by 
\[
d(n, m)=
\begin{cases}
a(n)\lor a(m) & \text{if $n\neq m$}; \\
0 & \text{if $n=m$}. 
\end{cases}
\]
Then 
$d\in \ult(M, S)$ 
and 
$d$ is non-complete.
In particular, 
$\{n\}_{n\in\nn}$ is 
Cauchy in $(M, d)$, and it does not have  any limit point in $(M, d)$. 
\end{proof}

\begin{proof}[Proof of Corollary \ref{cor:ultcpt}]
Assume that $X$ is not compact. 
Then there exists a closed countable discrete subset 
$M$ of $X$. 
By Theorem \ref{thm:ultex} and Lemma \ref{lem:ultcb}, 
we obtain a non-complete 
$S$-valued  ultrametric $D$ 
on $X$ with 
$D\in \ult(X, S)$. 
This implies Corollary \ref{cor:ultcpt}. 
\end{proof}

%%%%%%%%%%%%%%%%%%%%%%%%%%%%%%%%%%%%%%%%%%%%%%%%%%%%%%%%%%%%%%%%%%%%%%%%%%%%%%%%%%%%%%%%%%%%%%%%%%%%%%%%%%%%%%%%%%%%%%%%%%%%%%%%%%%%%%%%%%%%%%%%%%%%%%%%%%%%%%%%%%%%%%%%%%%%%%%%%%%%%%%%%%%%%%%%%%%%%%%%%%%%%%%%%%%%%%%%%%%%%%%%%%%%%%%%%%%%%%%%%%%%%%%%%%%%%%%%%%%%%%%%%%%%%%%%%%%%%%%%%%%%%%%%%%%%%%%%%%%%%%%%%%%%%%%%%%%%%%%%%%%%%%%%%%%%%%%%%%%%%%%%%%%%%%%

\section{An interpolation theorem of ultrametrics}\label{sec:ui}

In this section, we prove Theorem \ref{thm:ultint}.

\subsection{Amalgamations}

The following lemma is a specialized version of 
\cite[Proposition 3.2]{IsI}
for our study on $S$-valued ultrametrics. 
\begin{lem}\label{lem:ultamal111}
Let $S$ be a range set possessing at least two elements. 
Let 
$(X, d_X)$ 
and 
$(Y, d_Y)$ 
be $S$-valued ultrametric spaces, 
and let 
$Z=X\cap Y$. 
Assume that 
\begin{enumerate}
\renewcommand{\labelenumi}{(\Alph{enumi})}
\item $Z\neq \emptyset$;
\item $d_X|_{Z^2}=d_Y|_{Z^2}$;
\item there exists 
$s\in S_+$ 
such that for every 
$x\in X\setminus Z$ 
we have 
$\inf_{z\in Z}d_X(x, z)=s$. 
\end{enumerate}
Then there exists an $S$-valued ultrametric $h$ on 
$X\cup Y$ 
such that 
\begin{enumerate}
	\item $h|_{X^2}=d_X$; 
	\item $h|_{Y^2}=d_Y$. 
\end{enumerate}
\end{lem}
\begin{proof}
We define a symmetric function 
$h: (X\cup Y)^2\to [0, \infty)$ 
by
\begin{align*}
	h(x, y)=
		\begin{cases}
		d_X(x, y) & \text{if $x, y\in X$;}\\
		d_Y(x, y) & \text{if $x, y\in Y$;}\\
		\inf_{z\in Z}(d_X(x, z)\lor d_Y(z, y)) & \text{if $(x, y)\in X\times Y$. }
		\end{cases}
\end{align*}
Since 
$d_X|_{Z^2}=d_Y|_{Z^2}$, 
the function 
$h$ is well-defined. 
By the definition, 
$h$ satisfies the conditions $(1)$ and $(2)$. 

We next prove that $h$ satisfies the strong triangle inequality. 
In the case where 
$x, y\in X$ 
and 
$z\in Y$, 
for all 
$a, b\in Z$ 
 we have 
\begin{align*}
	h(x, y)&=d_X(x, y)\le d_X(x, a)\lor d_X(a, b)\lor d_X(b, y)\\
	&=d_X(x, a)\lor d_Y(a, b)\lor d_X(b, y)\\
	&\le (d_X(x, a)\lor d_Y(a, z))\lor (d_Y(z, b)\lor d_X(b, y)), 
\end{align*}
and hence we obtain 
$h(x, y)\le h(x, z)\lor h(z, y)$. 
In the case where  
$x, z\in  X$ and $y\in Y$, 
for all 
$a\in Z$ 
we have 
\begin{align*}
	h(x, y)&\le d_X(x, a)\lor d_Y(a, y)\\
	&\le d_X(x, z)\lor (d_X(z, a)\lor d_Y(a, y)), 
\end{align*}
and hence we have 
$h(x, y)\le h(x, z)\lor h(z, y)$. 
By replacing the role of $X$ with that of $Y$, 
we see that 
$h$ satisfies the strong triangle inequality. 
\par

We now prove that $h$ takes values in $S$. 
It suffices to show that 
for all  
$x\in X\setminus Z$ 
and 
$y\in Y\setminus Z$,  
we have 
$h(x, y)\in S$. 
By the assumption (C) and the definition of $h$, 
we obtain 
$s \le h(x, y)$. 
If 
$s=h(x, y)$, 
then 
$h(x, y)$ is in $S$. 
If 
$s<h(x, y)$, 
by the assumption (C), 
there exists 
$z\in Z$ 
with   
$h(x, z)<h(x, y)$. 
Lemma \ref{lem:isosceles} implies that 
$h(x, y)=h(z, y)$. 
Since 
$h(z, y)=d_Y(z, y)$, 
we have 
$h(x, y)\in S$. 
This completes the proof. 
\end{proof}

Let $X$ and $Y$ be two sets,  
and let $\tau: X\to Y$ be a bijective map. 
For a metric $d$  on $Y$,  
we denote by 
$\tau^*d$ 
the metric on $X$ 
defined by 
$(\tau^{*}d)(x, y)=d(\tau(x), \tau(y))$. 
Remark that the map 
$\tau$ is an isometry from  
$(X, \tau^*d)$ 
into  
$(Y, d)$. 

The following  Proposition \ref{prop:ultGH} and Lemmas \ref{lem:ultam} and \ref{lem:key} are  ultrametric 
versions of \cite[Proposition 3.1, Lemma 3.4, Lemma 3.5]{IsI}.

\begin{prop}\label{prop:ultGH}
Let $S$ be a range set possessing at least two elements. 
Let $X$ be an $S$-valued  ultrametrizable space. 
Let   
$r\in S_+$ and 
$d, e\in \ult(X, S)$ 
satisfy the inequality 
$\mathcal{UD}_X^S(d, e)\le r$. 
Put 
$X_0=X$, 
and 
let 
$X_1$ 
be a set with 
$\card(X_1)=\card(X_0)$ 
and 
$X_0\cap X_1=\emptyset$. 
Let 
$\tau :X_0\to X_1$ 
be a bijection.  
Then there exists an ultrametric 
$h\in \ult(X_0\sqcup X_1, S)$ 
such that 
\begin{enumerate}
	\item $h|_{X_0^2}=d$;
	\item $h|_{X_1^2}=(\tau^{-1})^*e$; 
	\item for every 
	$x\in X_0$ 
	we have 
	$h(x, \tau(x))= r$. 
\end{enumerate}
\end{prop}
\begin{proof}
We define a symmetric function 
$h: (X_0\sqcup X_1)^2\to [0, \infty)$ 
by
\begin{align*}
	h(x, y)=
		\begin{cases}
		d(x, y) & \text{if $x, y\in X_0$;}\\
		 (\tau^{-1})^*e(x, y) & \text{if $x, y\in X_1$;}\\
		 \inf_{a\in X_0}
		 (d(x, a)\lor r\lor (\tau^{-1})^*e(\tau(a), y)) & \text{if $(x, y)\in X_0\times X_1$. }
		\end{cases}
\end{align*}
By the definition, 
for every 
$x\in X$, 
we have 
$h(x, \tau(x))\ge r$,  
and
	\[
	h(x, \tau(x))\le d(x, x)\lor r\lor (\tau^{-1})^*e(\tau(x), \tau(x))=r. 
	\]
Therefore for every 
$x\in X$ 
we have 
$h(x, \tau(x))=r$.  

We now prove that $h$ satisfies the strong triangle inequality. 
In the case where 
$x, y\in X_0$ 
and 
$z\in X_1$, 
for all 
$a, b\in X_0$, by $\mathcal{UD}_X^S(d, e)\le r$
we have 
\begin{align*}
	&h(x, y)= d(x, y)\le d(x, a)\lor d(a, b)\lor d(b, y)\\
	&\le  d(x, a)\lor r\lor (\tau^{-1})^*e(\tau(a), \tau(b))\lor d(b, y)\\
	&\le d(x, a)\lor r\lor (\tau^{-1})^*e(\tau(a), z)\lor (\tau^{-1})^*e(\tau(b), z)\lor d(b, y)\\
	& \le  (d(x, a)\lor r\lor (\tau^{-1})^*e(\tau(a), z))\lor (d(y, b)\lor r\lor (\tau^{-1})^*e(\tau(b), z)), 
\end{align*}
and hence we obtain
$h(x, y)\le h(x, z)\lor h(z, y)$. 
In the case where 
$x, z\in X_0$ 
and 
$y\in X_1$, 
for all 
$a\in X_0$
we have 
\begin{align*}
	h(x, y)&\le d(x, a)\lor r\lor (\tau^{-1})^*e(\tau(a), y)\\
	&\le 
	d(x, z)\lor (d(z, a)\lor r\lor (\tau^{-1})^*e(\tau(a), y)), 
\end{align*}
and hence  
$h(x, y)\le h(x, z)\lor h(z, y)$. 
By replacing the role of $X_0$ with that of $X_1$, 
we see that $h$ satisfies the strong triangle inequality. 
By the property (3), we also see that $h\in \ult(X_0\sqcup X_1)$. 
\par

We next prove that $h$ takes values in $S$. 
It suffices to show that for all $(x, y)\in X_0\times X_1$, we have 
$h(x, y)\in S$. 
By the definition of $h$, we have $r\le h(x, y)$. 
If $r=h(x, y)$, then $h(x, y)$ is in $S$. 
If $r<h(x, y)$, by $h(x, \tau(x))=r$, we have $h(x, \tau(x))<h(x, y)$. 
From Lemma \ref{lem:isosceles}, it follows that 
$h(x, y)=h(\tau(x), y)$. 
Since $h(\tau(x), y)= (\tau^{-1})^*e(\tau(x), y)$ and $ (\tau^{-1})^*e(\tau(x), y)\in S$, 
we conclude that $h$ takes values in $S$. 
\end{proof}

\begin{lem}\label{lem:ultam}
Let $S$ be a range set possessing at least two elements, and let $s\in S_+$. 
Let 
$\{(A_i,  e_i)\}_{i\in I}$ 
be a mutually disjoint family of 
$S$-valued ultrametric spaces. 
Then there exists an ultrametric 
$h\in \ult(\coprod_{i\in I}A_i, S)$ 
such that 
\begin{enumerate}
\item for every $i\in I$ we have $h|_{A_i^2}=e_i$; 
\item for all distinct $i, j\in I$, and for all $x\in A_i$ and $y\in A_j$, 
we have $s\le h(x, y)$. 
\end{enumerate}
\end{lem}
\begin{proof}
We may assume that  $I=\kappa$, where $\kappa$ is a cardinal. 
By transfinite induction, 
we define a desired  ultrametric $h$ as follows:
Let 
$\alpha<\kappa+1$. 
Assume that we already define ultrametrics 
$\{h_{\gamma}\}_{\gamma<\alpha}$ 
such that 
\begin{enumerate}
	\item if 
	$\gamma<\delta<\alpha$, 
	then for all 
	$x, y\in A_{\gamma}$ 
	we have 
	$h_{\gamma}(x, y)=h_{\delta}(x, y)$; 
	\item for every 
	$\beta<\alpha$ 
	we have 
	$h_{\beta}\in \ult(\coprod_{\gamma<\beta}A_{\gamma}, S)$; 
	\item if $\gamma, \delta, \beta<\alpha$ satisfy $\gamma< \delta<\beta$, and  if 
	$x\in A_{\gamma}$ 
	and 
	$y\in A_{\delta}$, 
	then
	we have $s\le h_{\beta}(x, y)$. 
\end{enumerate} 
If 
$\alpha=\beta+1$ for some $\beta$, then 
we can define an $S$-valued  ultrametric 
$h_a\in \ult(\coprod_{\gamma<\alpha}A_{\gamma}, S)$ 
by using  Proposition \ref{prop:ultamal2} for 
$X=\coprod_{\gamma<\beta}A_{\gamma}$, 
$Y=A_{\alpha}$ 
and  
$r=s$. 
Assume next that  $\alpha$ is a limit ordinal.  
We define a function  $h_{\alpha}$ on 
$\left(\coprod_{\gamma<\alpha}A_{\gamma}\right)^2$ by 
	\[
	h_{\alpha}(x, y)=h_{\delta}(x, y), 
	\]
where 
$\delta<\alpha$ 
is the first ordinal with 
$x, y\in \coprod_{\epsilon<\delta}A_{\epsilon}$. 
By the inductive hypothesis $(1)$, 
the function 
$h_{\alpha}$ 
is well-defined. 
From the inductive  hypotheses $(2)$ and $(3)$, 
it follows that 
$h_{\alpha}\in \ult(\coprod_{\gamma<\alpha}A_{\gamma}, S)$. 
Put 
$h=h_{\kappa}$, 
then the proof is completed. 
\end{proof}

\begin{lem}\label{lem:key}
Let $S$ be a range set possessing at least two elements. 
Let $X$ be an $S$-valued  ultrametrizable space, 
and let 
$\{A_i\}_{i\in I}$ 
be a discrete family of closed subsets of $X$. 
Let 
$d\in \ult(X, S)$,  
and let
$\{e_i\}_{i\in I}$ 
be a family of ultrametrics such that
$e_i\in \ult(A_i, S)$. 
Assume that 
$\sup_{i\in I} \mathcal{UD}_{A_i}^S(e_{A_i}, d|_{A_i^2})<\infty$.  
Let $\eta$ be a member in $S_+$ such that 
	\[
	\sup_{i\in I} \mathcal{UD}_{A_i}^S(e_{A_i}, d|_{A_i^2})\le \eta. 
	\]
Let 
$\{B_i\}_{i\in I}$ 
be a mutually disjoint family of sets 
such that for all 
$i\in I$ 
we have 
$\card(B_i)=\card(A_i)$ 
and 
$X\cap B_i=\emptyset$. 
Let 
$\tau: \coprod_{i\in I}A_i\to \coprod_{i\in I}B_i$ 
be a bijection such that 
for each 
$i\in I$ 
the map 
$\tau_i=\tau|_{A_i}$ 
is a bijection between 
$A_i$ 
and 
$B_i$. 
Then 
there exists an $S$-valued  ultrametric $h$ on 
$X\sqcup\coprod_{i\in I}B_i$ 
such that
\begin{enumerate}
	\item for every 
	$i\in I$ 
	we have 
	$h|_{B_i^2}=(\tau_i^{-1})^*e_i$; 
	\item $h|_{X^2}=d$; 
	\item for every 
	$x\in \coprod_{i\in I}A_i$ 
	we have $h(x, \tau(x))= \eta$. 
\end{enumerate}
\end{lem}
\begin{proof}
By Proposition \ref{prop:ultGH}, 
for every
$i\in I$, 
we find  an $S$-valued ultrametric 
$l_i\in \ult(A_i\sqcup B_i, S)$ 
such that 
\begin{enumerate}
	\item $l_i|_{A_i^2}=d|_{A_i^2}$;
	\item $l_i|_{B_i^2}=(\tau_i^{-1})^*e_i$; 
	\item for every $x\in A_i$ 
	we have $l_i(x, \tau(x))= \eta$. 
	\end{enumerate}
By Lemma \ref{lem:ultam}, 
we obtain an 
$S$-valued ultrametric 
$k$ which is a member of 
$\ult(\coprod_{i\in I}(A_i\sqcup B_i), S)$ 
such that
\begin{enumerate} 
\item for each $i\in I$ 
we have 
$k|_{(A_i\sqcup B_i)^2}=l_i$;
\item for all distinct $i, j\in I$, 
and for all 
$x\in A_i\sqcup B_i$ 
and 
$y\in A_j\sqcup B_j$, 
we have 
$\eta\le h(x, y)$.

\end{enumerate}
Since 
\[
X\cap \left(\coprod_{i\in I}(A_i\sqcup B_i)\right)=\coprod_{i\in I}A_i,
\] 
and since the ultrametric 
$k$ satisfies the assumptions stated in  Lemma \ref{lem:ultamal111}, 
we obtain an 
$S$-valued ultrametric $h$ on 
$X\sqcup \coprod_{i\in I}B_i$ 
such that 
\begin{enumerate}
	\item $h|_{X^2}=d$; 
	\item $h|_{(\coprod_{i\in I}B_i)^2}=k|_{(\coprod_{i\in I}B_i)^2}$. 
\end{enumerate}
By the definitions of ultrametrics $l_i$ and $k$, 
we conclude that $h$ is an $S$-valued ultrametric as required. 
\end{proof}

\subsection{Proof of Theorem \ref{thm:ultint}}
Before proving Theorem \ref{thm:ultint}, 
we recall:
\begin{prop}\label{prop:disclosed}
Let $T$ be a topological space,  
and 
let 
$\{S_i\}_{i\in I}$ 
be a discrete family of closed subsets of $T$.
Then 
$\bigcup_{i\in I}S_i$ 
is closed in $T$. 
\end{prop}

In the proof of \cite[Theorem 1.1]{IsI}, 
the author used the Michael 
continuous selection theorem for paracompact spaces. 
Instead of that continuous selection theorem, 
to prove Theorem \ref{thm:ultint}, 
we now use 
the $0$-dimensional Michael continuous selection theorem 
(Theorem \ref{thm:zeroMichael}). 
\begin{proof}[Proof of Theorem \ref{thm:ultint}]
Let $C\in [1, \infty)$, 
and let $S$ be a $C$-quasi-complete range set. 
Let $X$ be an $S$-valued  ultrametrizable space. 
Let
$\{A_i\}_{i\in I}$ 
be a discrete family of closed subsets of $X$. 
Let   $d\in \ult(X, S)$,   
and let 
$\{e_i\}_{i\in I}$ be a
family 
of  $S$-valued ultrametrics with 
$e_i\in \ult(A_i, S)$. 
We may assume that $S$ possesses at least two elements. 
\par

If 
$\sup_{i\in I}\mathcal{UD}_{A_i}^S(e_{i}, d|_{A_i^2})=\infty$, 
then Theorem \ref{thm:ultint} follows from Lemma \ref{lem:ultam} 
and Theorem \ref{thm:ultex}. 
We may assume  that
$\sup_{i\in I}\mathcal{UD}_{A_i}^S(e_{i}, d|_{A_i^2})<\infty$.
\par

Let $\eta$ be a member in $S$ such that 
\[
\sup_{i\in I}\mathcal{UD}_{A_i}^S(e_{i}, d|_{A_i^2})\le \eta\le C\cdot  \sup_{i\in I}\mathcal{UD}_{A_i}^S(e_{i}, d|_{A_i^2}).
\] 
Let $\{B_i\}_{i\in I}$,  
and let
$\tau: \coprod_{i\in I}A_i\to \coprod_{i\in I}B_i$ 
be the same family and the same map as in Lemma \ref{lem:key}, respectively. 
Put 
$Z=X\sqcup\coprod_{i\in I}B_i$. 
By Lemma \ref{lem:key},  
we find an 
$S$-valued ultrametric $h$ on $Z$ such that
\begin{enumerate}
	\item for every $i\in I$ we have $h|_{B_i^2}=(\tau_i^{-1})^*e_i$; 
	\item $h|_{X^2}=d$; 
	\item for every 
	$x\in \coprod_{i\in I}A_i$ 
	we have 
	$h(x, \tau(x))= \eta$. 
\end{enumerate}
By Theorem \ref{thm:ultemb}, 
we can take an isometric embedding 
$H$ 
from 
$(Z, h)$ 
into a complete 
$S$-valued  ultra-normed $\zz$-module
$(Y, D_Y)$. 
Define a map 
$\phi: Z\to \clo(Y)$ 
by 
$\phi(x)=B(H(x), \eta)$. 
By Corollary \ref{cor:ultlsc}, 
the map $\phi$ is lower semi-continuous. 
We define a map 
$f:\bigcup_{i\in I}A_i\to Y$ by $f_i(x)=H(\tau(x))$. 
Then $f$ is continuous. 
By the property  
$(3)$ 
of 
$h$,  
for every 
$x\in \bigcup_{i\in I}A_i$ 
we have 
$f(x)\in \phi(x)$. 
\par

Since $(Y, D_Y)$ is complete, 
 we can  apply  
 the $0$-dimensional Michael continuous selection  
 theorem (Theorem \ref{thm:zeroMichael})
to the map $f$, 
and hence
we obtain a continuous map
$F:X\to Y$ 
such that 
$F|_{\bigcup_{i\in I}A_i}=f$ 
and for every 
$x\in X$ 
we have 
$F(x)\in \phi(x)$. 
Note that 
$F(x)\in \phi(x)$ 
means that 
$D_Y(F(x), H(x))\le \eta$. 

By Lemma \ref{lem:ultam},  
we obtain an ultrametric 
$k\in \ult(\coprod_{i\in I}A_i, S)$ 
such that for every 
$i\in I$ 
we have 
$k|_{A_i^2}=e_i$. 
Since the 
$S$-valued ultrametric $k$ generates the same topology as
$\coprod_{i\in I}A_i$, and since 
$\coprod_{i\in I}A_i$ 
is closed in $X$ (see Proposition \ref{prop:disclosed}), 
we can apply Theorem \ref{thm:ultex} to the $S$-valued ultrametric $k$,  
and hence  there exists an $S$-valued  ultrametric 
$r\in \ult(X, S)$ 
such that for every 
$i\in I$ 
we have  
$r|_{A_i^2}=e_i$. 
Put 
$l=\min\{r, \eta\}$. 
Note that by Lemma \ref{lem:ultland}, 
we have  $l\in \ult(X, S)$. 
\par

Put 
$D=D_Y\times_{\infty} l$.
Then $D$ is an $S$-valued  ultrametric on 
$Y\times X$.  
Take a base point 
$o\in X$. 
Define a map 
$E: X\to Y\times X$ 
by 
	\[
	E(x)=(F(x), x), 
	\]
Since the second component of $E$ is a topological embedding, 
so is $E$.  
\par

We also define a map 
$K: X\to Y\times X$ 
by 
	\[
	K(x)=(H(x), o). 
	\]
Then, 
by the definition of the ultrametric $D$ on
$Y\times X$,  
the map 
$K$ 
from 
$(X, d)$ 
to 
$(Y\times X, D)$ 
is an isometric embedding. 
Since for every 
$x\in X$ 
we have 
$D_Y(F(x), H(x))\le \eta$ 
and  
$l(x, o)\le \eta$, 
we obtain 
	\[
	D(E(x), K(x))=D_Y(F(x), H(x))\lor l(x, o)\le \eta.
	\] 
	\par

Define a function 
$m:X^2\to [0, \infty)$ 
by 
$m(x, y)=D(E(x), E(y))$, 
then $m$ is an $S$-valued  ultrametric on $X$. 
Since $E$ is a topological embedding, 
we see that 
$m\in \ult(X, S)$. 
For every 
$i\in I$,  
and for all 
$x, y\in A_i$, 
we have 
$D_Y(F(x), F(y))=e_i(x, y)$ 
and 
	\[
	l(x, y)\le r(x, y)=e_i(x, y);
	\] 
thus we obtain 
	\[
	D(E(x), E(y))=D_Y(F(x), F(y)) \lor l(x, y)=e_i(x, y), 
	\]
and hence 
$m|_{A_i^2}=e_i$. 
Moreover,  
we have 
\[
\sup_{i\in I}\mathcal{UD}_{A_i}^S(e_{i}, d|_{A_i^2})\le \mathcal{UD}_X^S(m, d).
\] 
We also obtain the inequality 
$\mathcal{UD}_X^S(m, d)\le \eta$;
indeed, 
for all 
$x, y\in X$, 
\begin{align*}
m(x, y)&=D(E(x), E(y))\\
& \le D(E(x), K(x))\lor D(K(x), K(y))\lor D(K(y), E(y)) \\
&\le D(K(x), K(y))\lor \eta=d(x, y)\lor \eta,  
\end{align*}
and 
\begin{align*}
d(x, y)&=D(K(x), K(y))\\ 
& \le D(K(x), E(x))\lor D(E(x), E(y))\lor D(E(y), K(y)) \\
&\le D(E(x),E(y))\lor \eta= m(x, y)\lor \eta. 
\end{align*}
Therefore $\mathcal{UD}_X^S(m, d)\le \eta$, and hence we conclude that 
\[
\sup_{i\in I}\mathcal{UD}_{A_i}^S(e_{i}, d|_{A_i^2})\le 
\mathcal{UD}_X^S(m, d)\le 
C\cdot \sup_{i\in I}\mathcal{UD}_{A_i}^S(e_{i}, d|_{A_i^2}).
\] 
This completes the proof of the  former 
part of Theorem \ref{thm:ultint}. 
\par

By the latter part of Theorem \ref{thm:ultex}, 
we can choose $l$ as a complete $S$-valued  ultrametric. 
Then $m$ becomes a complete $S$-valued  ultrametric. 
This leads to the proof of the  latter part of 
Theorem \ref{thm:ultint}. 
\end{proof}

In Theorem \ref{thm:ultint},  
by letting $I$ be a singleton, 
we obtain the following:
\begin{cor}\label{cor:ultmetext}
Let $C\in [1, \infty)$, and let $S$ be a $C$-quasi-complete range set. 
Let $X$ be an $S$-valued  ultrametrizable space,  
and let $A$ be a closed subset of $X$.
Then  
for every  
$d\in \ult(X, S)$, 
and for every  
$e\in \ult(A, S)$, 
there exists  an ultrametric 
$m\in \ult(X, S)$ 
satisfying the following:
\begin{enumerate}
	\item $m|_{A^2}=e$; 
	\item $\mathcal{UD}_A^S(e, d|_{A^2})\le \mathcal{UD}_X^S(m, d)\le 
	C\cdot \mathcal{UD}_A^S(e, d|_{A^2})$. 
\end{enumerate}
Moreover, 
if $X$ is completely metrizable,  
and if 
$e\in \ult(A, S)$ 
is  a complete $S$-valued ultrametric, 
then we can choose 
$m\in \ult(X, S)$ as a complete metric. 
\end{cor}
%%%%%%%%%%%%%%%%%%%%%%%%%%%%%%%%%%%%%%%
%%%%%%%%%%%%%%%%%%%%%%%%%%%%%%%%%%%%%%%
%%%%%%%%%%%%%%%%%%%%%%%%%%%%%%%%%%%%%%%
%%%%%%%%%%%%%%%%%%%%%%%%%%%%%%%%%%%%%%%
%%%%%%%%%%%%%%%%%%%%%%%%%%%%%%%%%%%%%%%
%%%%%%%%%%%%%%%%%%%%%%%%%%%%%%%%%%%%%%%
%%%%%%%%%%%%%%%%%%%%%%%%%%%%%%%%%%%%%%%
%%%%%%%%%%%%%%%%%%%%%%%%%%%%%%%%%%%%%%%

\section{Transmissible properties and ultrametrics}\label{sec:ulttrans}
In this section,  
we introduce the transmissible property, 
originally defined 
in \cite{IsI}, 
and we prove Theorem \ref{thm:ulttrans} concerning dense $G_{\delta}$ subsets of 
spaces of ultrametrics.

\subsection{Transmissible properties on metric spaces}\label{subsec:61}

Let  
$\mathcal{P}^*(\nn)$ 
be  the set of all non-empty subsets of  $\nn$. 
For a  topological space $T$, 
we denote by 
$\mathcal{F}(T)$ 
the set of all closed subsets of $T$. 
For a subset 
$W\in \mathcal{P}^*(\nn)$, 
and for a  set $E$, 
we denote by 
$\seq(W, E)$ 
the set of all finite injective sequences 
$\{a_i\}_{i=1}^n$ 
in  
$E$ with 
$n\in W$.

\begin{df}[\cite{IsI}]\label{def:transp}
Let $Q$ be an at most countable set, 
$P$ a topological space. 
Let 
$F: Q\to \mathcal{F}(P)$ 
and 
$G: Q\to \mathcal{P}^*(\nn)$ 
be maps.
Let $Z$ be a set.  
Let $\phi$ be a correspondence 
assigning a pair 
$(q, X)$  of 
$q\in Q$ 
and a metrizable space $X$ to a map 
$\phi^{q, X}:\seq(G(q), X)\times Z\times \met(X) \to P$. 
We say that 
a sextuple 
$(Q, P, F, G, Z, \phi)$ 
is 
a \emph{transmissible parameter}
if 
for every metrizable space $X$, for every $q\in Q$, 
and
 for every $z\in Z$  the following are satisfied:
\begin{enumerate} 
	\item[(TP1)]\label{item:tp1}
	for every 
	$a\in \seq(G(q), X)$ 
	the map 
	$\phi^{q, X}(a, z):\met(X)\to P$ 
	defined by 
	$\phi^{q, X}(a, z)(d)=\phi^{q, X}(a, z, d)$ 
	is continuous, 
	where $\met(X)$ is equipped with the topology 
	induced from $\mathcal{D}_X$; 
	\item[(TP2)]\label{item:tp2}
	 for every  
	 $d\in \met(X)$, 
	if  $S$ is a subset of $X$ and 
	$a\in \seq(G(q), S)$, 
	 then we have 
	$\phi^{q, X}(a, z, d)=\phi^{q, S}(a, z, d|_{S^2})$. 
	\end{enumerate}
\end{df}
We introduce a property determined by a transmissible parameter. 

\begin{df}[\cite{IsI}]\label{def:transpro}
Let 
$\mathfrak{G}=(Q, P, F, G, Z, \phi)$ 
be a transmissible parameter. 
Let $(X, d)$ be a metric space. 
We say that  $(X, d)$ satisfies 
the \emph{$\mathfrak{G}$-transmissible  property} if 
there exists  
$q\in Q$ 
such that for every 
$z\in Z$ 
and for every 
$a\in \seq(G(q), X)$ 
we have 
$\phi^{q, X}(a, z, d)\in F(q)$. 
We say that 
$(X, d)$ 
satisfies the \emph{anti-$\mathfrak{G}$-transmissible property} if 
$(X, d)$  satisfies the negation of the $\mathfrak{G}$-transmissible property; namely, 
for every 
$q\in Q$ 
there exist 
$z\in Z$ 
and 
$a\in \seq(G(q), X)$ 
with 
$\phi^{q, X}(a, z, d)\in P\setminus F(q)$. 
A property on metric spaces is  
a \emph{transmissible property} 
(resp.~\emph{anti-transmissible property})
if it is equivalent to the $\mathfrak{G}$-transmissible property 
(resp.~anti-$\mathfrak{G}$-transmissible property) 
for some
transmissible parameter $\mathfrak{G}$. 
\end{df}

By the condition (TP2) in Definition \ref{def:transp}, 
we obtain the following:
\begin{lem}\label{lem:here}
Let $\mathfrak{G}$ be a transmissible parameter. 
If a metric space 
$(X, d)$ 
satisfies the 
$\mathfrak{G}$-transmissible property, 
then so does every metric subspace of 
$(X, d)$. 
\end{lem}

\subsection{Transmissible properties on ultrametric spaces}\label{subsec:62}

The following concept is an 
$S$-valued ultrametric version of the singularity of 
the transmissible parameters
(see  \cite[Definition 1.3]{IsI}). 
\begin{df}\label{def:ultsing}
Let $S$ be  a range set. 
Let 
$\mathfrak{G}=(Q, P, F, G, Z, \phi)$ 
be a transmissible parameter. 
We say that 
$\mathfrak{G}$ 
is 
\emph{$S$-ultra-singular} 
if for each 
$q\in Q$ 
and for every 
$\epsilon\in (0, \infty)$
there exist 
$z\in Z$,  a finite $S$-valued ultrametric  space 
$(R, d_{R})$, and 
an index $R=\{r_i\}_{i=1}^{\card(R)}$
such that 
\begin{enumerate}
	\item 
	$\delta_{d_{R}}(R)\le \epsilon$;
	\item 
	$\card(R)\in G(q)$; 
	\item 
	$\phi^{q, R}\left(\{r_i\}_{i=1}^{\card(R)}, z, d_R\right)\in P\setminus F(q)$. 
\end{enumerate}
\end{df}

By the definitions of 
$\mathcal{D}_X$ 
and 
$\mathcal{UD}_X^S$, 
we obtain:
\begin{lem}\label{lem:ultandmet}
Let $S$ be a range set. 
For every $S$-valued  ultrametrizable space $X$,  
and for all $d, e\in \ult(X, S)$
 we have 
\[
\mathcal{D}_X(d, e)\le \mathcal{UD}_X^S(d, e).
\] 
In particular, the identity map 
\[
1_{\ult(X, S)}:(\ult(X, S), \mathcal{UD}_X^S)
\to 
\left(\ult(X, S), \mathcal{D}_X|_{\ult(X, S)^2}\right)
\]
 is continuous. 
\end{lem}

Let $S$ be a range set. 
Let $X$ be an $S$-valued  ultrametrizable space, 
and let 
$\mathfrak{G}=(Q, P, F, G, Z, \phi)$ 
be a transmissible  parameter. 
For $q\in Q$, 
for 
$a\in \seq(G(q), X)$ 
and for
$z\in Z$, 
we denote by 
$US( X, S, \mathfrak{G}, q, a, z)$ 
the set of all 
$d\in \ult(X, S)$ 
such that 
$\phi^{q, X}(a, z, d)\in P\setminus F(q)$. 
We also denote by 
$US(X, S,  \mathfrak{G})$ 
the set of all 
$d\in \ult(X, S)$ 
such that 
$(X, d)$ 
satisfies the anti-$\mathfrak{G}$-transmissible property. 

\begin{prop}\label{prop:ultopen}
Let $S$ be a range set. 
Let $X$ be an ultrametrizable space,  and let 
$\mathfrak{G}=(Q, P, F, G, Z, \phi)$  be 
 a transmissible  parameter. 
Then 
for all 
$q\in Q$, 
$a\in \seq(G(q), X)$ and $z\in Z$, 
the set 
$US(X, S, \mathfrak{G}, q, a, z)$
is open in 
$(\ult(X, S), \mathcal{UD}_X^S)$. 
\end{prop}
\begin{proof}
Fix $q\in Q$, 
$a\in \seq(G(q), X)$ 
and 
$z\in Z$. 
Since the map 
$\phi^{q, X}(a, z): \met(X)\to P$ 
is continuous,  
Lemma \ref{lem:ultandmet} implies that
the map $\phi^{q, X}(a, z)|_{\ult(X, S)}: \ult(X, S)\to P$ is also continuous, 
where $\ult(X, S)$ is equipped with the topology induced from 
$\mathcal{UD}_X^S$. 
Since 
 \[
 US(X, S, \mathfrak{G}, q, a, z)=(\phi^{q, X}(a, z)|_{\ult(X, S)})^{-1}(P\setminus F(q)), 
 \]
the set $US(X, S, \mathfrak{G}, q, a, z)$ is open in 
 $(\ult(X, S), \mathcal{UD}_X^S)$. 
\end{proof}

%%%
\begin{cor}\label{cor:ultopen}
Let $S$ be a range set.  
Let 
$\mathfrak{G}=(Q, P, F, G, Z, \phi)$ 
be  a transmissible  parameter. 
Let $X$ be an $S$-valued  ultrametrizable space. 
Then 
the set 
$US(X, S,  \mathfrak{G})$ 
is 
$G_{\delta}$ in $\ult(X, S)$.
Moreover, 
if the set $Q$ is finite, 
then 
$US(X, S, \mathfrak{G})$ 
is open in 
$(\ult(X, S), \mathcal{UD}_X^S)$. 
\end{cor}
\begin{proof}
By the definitions of 
$US(X, S,  \mathfrak{G})$ 
and 
$US(X, S, \mathfrak{G}, q, a, z)$, 
we have 
\[
	US(X, S, \mathfrak{G})
	=
	\bigcap_{q\in Q}\bigcup_{a\in \seq(G(q), X)}\bigcup_{z\in Z}
	US(X, S, \mathfrak{G}, q, a, z). 
\]
This equality together with Proposition \ref{prop:ultopen} proves 
the lemma. 
\end{proof}

We say that a topological space is an \emph{$(\omega_0+1)$-space}
if it is homeomorphic to the one-point compactification of the countable 
discrete topological  space. 
\begin{lem}\label{lem:ultsing}
Let $S$ be a range set with the countable coinitiality. 
A transmissible parameter
 $\mathfrak{G}$ is $S$-ultra-singular  if and only if 
 there exists an 
$S$-valued ultrametric $(\omega_0+1)$-space 
with arbitrary small diameter 
satisfying the  anti-$\mathfrak{G}$-transmissible property. 
\end{lem}
\begin{proof}
Let 
$\mathfrak{G}=(Q, P, F, G, Z, \phi)$. 
First assume that there exists an 
$(\omega_0+1)$-ultrametric space 
with arbitrary small diameter 
satisfying the  anti-$\mathfrak{G}$-transmissible property. 
By the definition of anti-$\mathfrak{G}$-transmissible property, 
we see that $\mathfrak{G}$ is $S$-ultra-singular. 

Next assume that  $\mathfrak{G}$ is $S$-ultra-singular. 
Take a  strictly decreasing sequence $\{r(i)\}_{i\in \nn}$ with 
$\lim_{i\to \infty}r(i)=0$. 
Fix $\epsilon\in (0, \infty)$ and take a surjective map $\theta:\nn\to Q$. 
Take $N\in \nn$ such that for every $n>N$, we have $r(n)<\epsilon$. 
Then there exists a sequence 
$\{(R_i, d_i)\}_{i\in \nn}$ 
of finite ultrametric spaces  such that for each 
$i\in \nn$ 
there exist 
$z_i\in Z$ and an index $R_i=\{r_{i, j}\}_{j=1}^{\card(R_i)}$
satisfying
\begin{enumerate}
	\item $\delta_{d_{i}}(R_i)\le r(N+i)$; 
	\item $\card(R_i)\in G(\theta(i))$; 
	\item $\phi^{\theta(i), R_i}\left(\{r_{i, j}\}_{j=1}^{\card(R_i)}, z_i, d_i\right)\in P\setminus F(\theta(i))$. 
\end{enumerate}
Put
	\[
	L=\{\infty\}\sqcup \coprod_{i\in \nn}R_i, 
	\]
and define a metric 
$d_{L}$ 
on 
$L$ 
by 
\[
d_{L}(x, y)=
\begin{cases}
		d_{i}(x,y) & \text{if $x,y\in X_i$ for some $i$;}\\
		 r(N+i)\lor r(N+j) & \text{if $x\in X_i,y\in X_j$ for some $i\neq j$; }\\
		 r(N+i) & \text{if $x=\infty, y\in X_i$ for some $i$;}\\
		 r(N+i) & \text{if $x\in X_i, y=\infty$ for some $i$.}
	\end{cases}
\]
Note that 
this construction is a specific version of 
the telescope space defined in \cite{IsQ}. 
The $(L, d_L)$ 
is a metric
$(\omega_0+1)$-space with $\delta_{d_L}(L)\le \epsilon$. 
By the definition, 
the metric $d_L$ is an 
$S$-valued  ultrametric (see also \cite[Lemma 3.1]{IsQ}). 
By the properties $(2)$ and $(3)$ of 
$\{(R_i, d_i)\}_{i\in \nn}$, 
the metric space 
$(L, d_L)$ 
satisfies the 
anti-$\mathfrak{G}$-transmissible property. 
\end{proof}

Let $S$ be a range set. 
Let $\mathfrak{G}$ be a transmissible parameter. 
For a non-discrete $S$-valued  ultrametrizable space $X$, 
and for an $(\omega_0+1)$-subspace $R$ of $X$, 
we denote by 
$UT(X, S, R, \mathfrak{G})$ 
the set of all 
$d\in \ult(X, S)$
for which 
$(R, d|_{R^2})$ 
satisfies the anti-$\mathfrak{G}$-transmissible 
property.

Corollary \ref{cor:ultmetext} and 
Lemma \ref{lem:ultsing} imply  the following:
\begin{prop}\label{prop:ultdense}
Let $C\in [1, \infty)$, 
and let $S$ be a $C$-quasi-complete range set with 
the countable coinitiality. 
Let $\mathfrak{G}=(Q, P, F, G, Z, \phi)$ 
be an $S$-ultra-singular transmissible parameter. 
Then for every non-discrete $S$-valued  ultrametrizable space $X$, 
and for every 
$(\omega_0+1)$-subspace 
$R$ of $X$, 
the set 
$UT(X, S, R, \mathfrak{G})$ 
is dense in 
$(\ult(X, S), \mathcal{UD}_X^S)$.  
\end{prop}
\begin{proof}
Let $\epsilon\in (0, \infty)$ be an arbitrary number. 
Let $d\in \ult(X, S)$.  
Take an $(\omega_0+1)$ subspace $L$ of $R$ with 
$\delta_d(L)\le \epsilon$. 
By Lemma \ref{lem:ultsing}, 
there exists an $S$-valued  ultrametric $e\in \ult(L, S)$ with 
$\delta_{e}(L)\le \epsilon$ 
such that $(L, e)$ satisfies the 
anti-$\mathfrak{G}$-transmissible property.  
Since $\delta_d(L)\le \epsilon$ and $\delta_{e}(L)\le \epsilon$, 
by the definition of $\mathcal{UD}_L^S$ we have 
$\mathcal{UD}_L^S(d|_{L^2}, e)\le \epsilon$. 
By applying  Corollary \ref{cor:ultmetext} to $d$ and $e$, 
we obtain an $S$-valued ultrametric $m\in \ult(X, S)$ with 
\begin{enumerate}
\item $m|_{L^2}=e$;
\item $\mathcal{UD}_X^S(d, m)\le C\cdot\mathcal{UD}_L^S(d|_{L^2}, e)\le C\cdot \epsilon$. 
\end{enumerate}
By Lemma \ref{lem:here}, 
we see that $(X, m)$ satisfies the 
anti-$\mathfrak{G}$-transmissible property.  
Since $\epsilon$ is arbitrary and $m\in UT(X, S, R, \mathfrak{G})$, the proposition follows. 
\end{proof}

\begin{thm}\label{thm:ulttrans}
Let 
 $S$ be a quasi-complete range set with the countable coinitiality. 
Let 
$\mathfrak{G}$ 
be an $S$-ultra-singular transmissible parameter. 
Then for every non-discrete $S$-valued %%%%%%%%MMMM
 ultrametrizable space $X$, 
the set of all 
$d\in \ult(X, S)$ 
for which 
$(X, d)$ 
satisfies the
anti-$\mathfrak{G}$-transmissible property
 is dense 
 $G_{\delta}$ in
 the space  $(\ult(X, S), \mathcal{UD}_X^S)$.  
\end{thm}

\begin{proof}
Let $C\in [1, \infty)$, 
and let $S$ be a $C$-quasi-complete range set with 
the countable coinitiality. 
Let $X$ be a non-discrete metrizable space,  
and 
let 
$\mathfrak{G}$ 
be an $S$-ultra-singular transmissible parameter. 
Since $X$ is non-discrete, 
there exists an $(\omega_0+1)$-subspace 
$R$ of $X$. 
By the definitions, we have 
	\[
	UT(X, S,  R, \mathfrak{G})\subseteq US(X, S,  \mathfrak{G}). 
	\] 
From Proposition \ref{prop:ultdense} and Corollary \ref{cor:ultopen}, 
 it follows that 
 $US(X, S, \mathfrak{G})$ 
 is dense 
 $G_{\delta}$ 
 in 
 $(\ult(X, S), \mathcal{UD}_{X}^S)$. 
 This finishes the proof. 
\end{proof}

For a range set $S$, 
and 
for a complete metrizable space $X$, 
we denote by $\cu(X, S)$ the set of 
all complete metrics in $\ult(X, S)$. 
From  the latter part of Corollary \ref{cor:ultmetext}, 
we deduce the following:
\begin{thm}\label{thm:ultcm1}
Let $S$ be a quasi-complete range set with 
the countable coinitiality. 
Let 
$\mathfrak{G}$ 
be an $S$-ultra-singular transmissible parameter. 
For every non-discrete completely ultrametrizable space $X$, 
the set of all 
$d\in \cu(X, S)$ 
for which 
$(X, d)$ 
satisfies 
the 
anti-$\mathfrak{G}$-transmissible property 
is dense 
$G_{\delta}$ 
in  the ultrametric space 
$\left(\cu(X, S), \mathcal{UD}_X^S|_{\cu(X, S)^2}\right)$.  
\end{thm}
\begin{rmk}
Let $C\in [1, \infty)$, 
and let $S$ be a $C$-quasi-complete range set with 
the countable coinitiality. 
We can prove an 
$S$-valued ultrametric analogue of \cite[Proposition 4.15]{IsI}, 
which states that satisfying a metric inequality 
on metric spaces is a transmissible property. 
\end{rmk}

\subsection{Examples}\label{subsec:examples}
We show some examples of transmissible properties. 
\subsubsection{The doubling property}\label{subsec:doubling}
For a metic space 
$(X, d)$ 
and for a subset 
$A$ of $X$, 
we set  
	$\alpha_d(A)=\inf\{\, d(x, y)\mid \text{$x, y\in A$ and $x\neq y$}\, \}$. 
A metric space 
$(X, d)$ 
is said to be \emph{doubling} 
if there exist 
$C\in (0, \infty)$ 
and 
$\alpha\in (0, \infty)$ 
such that  for every finite subset 
$A$ of $X$  
we have 
	\[
	\card(A)\le C\left(\frac{\delta_d(A)}{\alpha_d(A)}\right)^{\alpha}. 
	\]
Note that 
$(X, d)$ 
is doubling if and only if 
$(X, d)$ 
has finite Assouad dimension 
(see e.g.,  \cite[Chapter 10]{H}).
\par

Similarly to \cite[Proposition 4.9]{IsI}, 
we obtain:
\begin{prop}
Let $S$ be a range subset with the countable coinitiality. 
The doubling property 
is a transmissible  property with an $S$-ultra-singular parameter. 
\end{prop}
\begin{proof}

Define a map 
$D: (\qq_{>0})^2\to \mathcal{F}((\rr_{> 0})^2)$
by 
\[
	D((q_1, q_2))
	=
	\{\, (x, y)\in (\rr_{> 0})^2\mid x\le q_1y^{q_2}\, \}, 
\]
and define a constant map 
$G_D: (\qq_{>0})^{2}\to \mathcal{P}(\nn)^*$
by 
$G_D(q)=[2, \infty)$.
Put 
$Z_D=\{1\}$. 
For each metrizable space $X$, 
and for each 
$q\in (\qq_{>0})^2$, 
define a map 
$\phi_{D}^{q, X}: \seq(G(q), X)\times Z_D\times \met(X)\to \rr$
by 
	\[
	\phi_{D}^{q, X}(\{a_i\}_{i=1}^N, 1, d)=
	\left(
	N, 
	\frac{\delta_{d}(\{\, a_i\mid i\in \{1, \dots, N\}\, \})}{\alpha_{d}(\{\, a_i\mid i\in \{1, \dots, N\}\, \})}
	\right). 
	\]
Let 
$\mathfrak{DB}=((\qq_{>0})^2, (\rr_{>0})^2, D, G_D, \{1\}, \phi_{D})$. 
Then 
$\mathfrak{DB}$ 
satisfies 
the condition(TP2) in Definition \ref{def:transp}, 
and
we see that $\mathfrak{DB}$ satisfies the condition (TP1). 
Hence $\mathfrak{DB}$ is a transmissible  parameter. 
The $\mathfrak{DB}$-transmissible  property is equivalent to 
the doubling property. 
We next prove that 
$\mathfrak{DB}$ 
is $S$-ultra-singular. 
For  
$q=(q_1, q_2)\in (\qq_{>0})^2$ 
and for 
$\epsilon\in S_+$, 
we denote by  
$(R_q, d_q)$ 
a finite metric space with 
$\card(R_q)>q_1+1$
such that $d_{q}(x, y)=\epsilon$ whenever $x\neq y$. 
Then, $(R_q, d_q)$ is an $S$-valued ultrametric space
$\delta_{d_q}(R_q)=\epsilon$,  
and 
	\[
	\phi^{q, R_q}(R_q, 1, d_q)=(\card(R_q), 1)\not\in D(q).
	\] 
This implies the proposition. 

\end{proof}

\subsubsection{The rich $S$-ultra-pseudo-cones property}
Let $(X, d)$ be a metric space. 
Let 
$\{A_i\}_{i\in \nn}$ 
be a sequence of subsets of $X$,  
and let  
$\{u_{i}\}_{i\in \nn}$ 
be a sequence in 
$(0, \infty)$. 
We say that a metric space 
$(P, d_P)$ 
is a
\emph{pseudo-cone of $X$} 
approximated by 
$(\{A_i\}_{i\in \nn}, \{u_i\}_{i\in \nn})$ 
if 
\[
	\lim_{i\to \infty}\mathcal{GH}((A_i, u_i\cdot d|_{A_i^2}), (P, d_P))=0
\]
(see \cite{IsO}), 
where $\mathcal{GH}$ is the Gromov--Hausdorff distance (see \cite{BBI}). 
For  a metric space 
$(X, d)$,
 we denote by 
 $\pc(X, d)$ 
 the class of all pseudo-cones of 
 $(X, d)$.
 Let $S$ be a range set, and let $T$ be a range subset of $S$ which is countable dense subset of $S$. 
Let 
$\mathscr{U}_T$ 
be the class of 
all finite ultrametric spaces on which  all distances are in $T$. 
We say that a metric space 
$(X, d)$ \emph{has rich $S$-ultra-pseudo-cones} if 
$\mathscr{U}_T$ is contained in $\pc(X, d)$ for some countable dense range subset  $T$ of $S$. 

\begin{lem}\label{lem:ultapprox}
 Let $S$ be a range set, 
 and let $T$ be a countable dense range  subset of $S$.  
Let $X$ be a finite discrete space, 
and let $d\in \ult(X, S)$. 
For every $\epsilon\in (0, \infty)$, 
 there exists a $T$-valued ultrametric $e\in \ult(X, T)$ such that 
for all $x, y\in X$ we have $|d(x, y)-e(x, y)|<\epsilon$. 
\end{lem}
\begin{proof}
Let $a_0, a_1, \dots, a_m$ be a sequence in $S$ with 
$\{\, d(x, y)\mid \text{$x, y\in X$}\, \}=\{ a_0, a_1, \dots, a_m\}$.  
We may assume that 
$a_0=0$ 
and 
$a_i<a_{i+1}$ for all $i$. 
Put $q_0=a_0 (=0)$. 
Since $T$ is dense in $S$, 
we can take  a sequence $q_1, \dots, q_m$ in $T_+$ such that 
$|a_i-q_i|<\epsilon$ 
and $q_{i}<q_{i+1}$ for all for all $i\in \{1, \dots, m\}$. 
Define a function 
$e:X\times X\to T$ by putting 
$e(x, y)=q_i$ if $d(x, y)=a_i$. 
Lemma \ref{lem:ultpreserving} implies  that $e$ is an ultrametric. 
By the definition,  
the ultrametric $e$ satisfies the conditions as required. 
\end{proof}
Since every compact ultrametric space has a finite 
$\epsilon$-net for all $\epsilon \in (0, \infty)$, 
Lemma \ref{lem:ultapprox} implies that for every range set $S$, 
every compact $S$-valued ultrametric space is arbitrarily approximated by
members of $\mathscr{U}_T$ 
in the sense of Gromov--Hausdorff 
for every countable dense range subset $T$ of $S$. 
Thus we have: 
\begin{cor}
Let $S$ be a range set. 
Let $(X, d)$ be an $S$-valued ultrametric space. 
The the following are equivalent to each other 
\begin{enumerate}
\item  $(X, d)$ has rich $S$-ultra-pseudo-cones 
\item $\pc(X, d)$ contains all compact $S$-valued ultrametric spaces.  
\item $\pc(X, d)$ contains $\mathscr{U}_T$ for all countable dense range 
subset $T$ of the range set $S$. 
\end{enumerate}
\end{cor}

A metric space is said to be 
\emph{rich pseudo-cones} 
if all compact metric spaces are pseudo-cones of it. 
In \cite[Proposition 4.12]{IsI}, 
it was proven that  the rich pseudo-cones property is 
an anti-transmissible property with a singular parameter. 
Similarly, 
we obtain:
\begin{prop}
Let $S$ be a range set with the countable coinitiality. 
The rich $S$-ultra-pseudo-cones property is 
an anti-transmissible property 
with an $S$-ultra-singular transmissible  parameter. 
\end{prop}

%%%%%%%%%%%%%%%%%%%%%%%%%%%%%%%%%%%%%%%
%%%%%%%%%%%%%%%%%%%%%%%%%%%%%%%%%%%%%%%
%%%%%%%%%%%%%%%%%%%%%%%%%%%%%%%%%%%%%%%
%%%%%%%%%%%%%%%%%%%%%%%%%%%%%%%%%%%%%%%
%%%%%%%%%%%%%%%%%%%%%%%%%%%%%%%%%%%%%%%
%%%%%%%%%%%%%%%%%%%%%%%%%%%%%%%%%%%%%%%
%%%%%%%%%%%%%%%%%%%%%%%%%%%%%%%%%%%%%%%
%%%%%%%%%%%%%%%%%%%%%%%%%%%%%%%%%%%%%%%
\section{Local transmissible properties and ultrametrics}\label{sec:ultloc}
In this section, 
we first investigate the basic properties on 
a specific ultrametric $u_S$ on a range set $S$. 
These properties help us to prove Lemma \ref{lem:ultBaire}. 
We also prove  Theorem \ref{thm:ultloc} 
which is a local version of Theorem \ref{thm:ulttrans}.
\par

We define an ultrametric $u_S$ on a range set $S$ in such a way that 
$u_S(x, y)$ is the infimum of 
$\epsilon\in (0, \infty)$ such that 
$x\le y\lor \epsilon$ and 
$y\le x\lor \epsilon$. 
Let 
$d_E$ denote the Euclidean metric on $S$ defined by 
$d_E(x, y)=|x-y|$. 

By the definition of $u_S$, we obtain:
\begin{lem}\label{lem:u-atai}
Let $S$ be a range set. 
Then for all distinct 
$x, y\in [0, \infty)$, 
we have 
$u_S(x, y)=x\lor y$. 
Hence 
$u_S$ is an $S$-valued ultrametric on $S$. 
\end{lem}

By the definitions of 
$d_E$ and 
$u_S$, we have:
\begin{lem}\label{lem:ookii}
Let $S$ be a range set. 
For all 
$a, b\in S$, 
we have 
\[
d_E(a, b)\le u_S(a, b).
\]
Moreover, 
the identity map 
$1_{S}:(S, u_S)\to (S, d_E)$ 
is 
continuous. 
\end{lem}

\begin{lem}\label{lem:ultcomp}
Let $S$ be a range set. Then 
the ultrametric space 
$(S, u_S)$ is complete. 
\end{lem}
\begin{proof}
Let 
$\{a_n\}_{n\in \nn}$ 
be a Cauchy sequence in 
$(S, u_S)$. 
Assume that  there exists 
$a\in S$ 
such that 
$\{\, n\in \nn\mid a_n=a\, \}$ 
is infinite. 
Since 
$\{a_n\}_{n\in \nn}$ 
is Cauchy, 
 it is convergent to $a$. 
Assume next that  for every 
$a\in S$,  
the set
$\{\, n\in \nn\mid a_n=a\, \}$ 
is finite. 
For every $\epsilon\in (0, \infty)$, 
we can take 
$N\in \nn$ such that  for all 
$n,m >N$, 
we have 
$u_S(a_n, a_m)\le  \epsilon$. 
By the assumption, for every $n\in N$, 
there exists $m>N$ with $a_n\neq a_m$. 
Thus by Lemma \ref{lem:u-atai}, we have 
$a_n\le \epsilon$. 
This implies that 
$\lim_{n\to \infty}a_n=0$. 
Therefore the space 
$(S, u_S)$ is complete. 
\end{proof}

Let $S$ be a range set. 
Let $H$ be a topological space, 
and let $C(H, S)$ be 
the set of all continuous function from $H$ into $S$, 
where $S$ is equipped with the Euclidean topology. 
We define an ultrametric 
$\mathcal{U}_H^S$ 
on $C(H, S)$ by 
\begin{equation*}\label{eq:U}
\mathcal{U}_H^S(f, g)=\min\left\{ 1,\sup_{x\in H}u_S(f(x), g(x)) \right\}. 
\end{equation*}
We also define a metric 
$\mathcal{E}_H$ 
on 
$C(H, [0, \infty))$ by 
\begin{equation*}\label{eq:E}
\mathcal{E}_H(f, g)=\min\left\{ 1,\sup_{x\in H}|f(x)-g(x)| \right\}. 
\end{equation*}
Note that 
$(C(H, [0, \infty)), \mathcal{E}_H)$
 is complete. 

\begin{rmk}
Let $S$ be a range set. 
The space 
$(\ult(X, S), \mathcal{UD}_{X}^S)$ 
and 
$(\met(X), \mathcal{D}_X$) 
can be considered as a topological subspace of   the spaces 
$(C(X^2, S), \mathcal{U}_{X^2}^S)$ 
and 
$(C(X^2, [0, \infty)), \mathcal{E}_{X^2})$, 
respectively. 
Namely, 
\begin{enumerate}
\item for every $S$-valued  ultrametrizable  space $X$, 
we have the inclusion 
$\ult(X, S)\subseteq C(X^2, S)$, 
and 
the metric 
$\mathcal{U}_{X^2}^S$ 
on 
$\ult(X, S)$ 
generates the same topology as  
that induced from 
$\mathcal{UD}_X^S$.  
\item  for every metrizable space $X$, 
we have 
$\met(X)\subseteq C(X^2, [0, \infty))$, 
and 
the metric 
$\mathcal{E}_{X^2}$ 
on 
$\met(X)$
generates the same topology as that induced from 
$\mathcal{D}_X$.  
\end{enumerate}
\end{rmk}

By the definitions of $\mathcal{E}_H$ and $\mathcal{U}_H^S$, and  by Lemma \ref{lem:ookii}, 
we have:
\begin{lem}\label{lem:functionookii}
Let $S$ be a range set. 
Let $H$ be a topological space. 
For all $f, g\in C(H, S)$, 
we have 
\[
\mathcal{E}_H(f, g)\le \mathcal{U}_H^S(f, g).
\]
Moreover, 
the inclusion map 
$(C(H, S), \mathcal{U}_H^S)\to (C(H, [0, \infty)), \mathcal{E}_H)$ 
is continuous. 
\end{lem}

Similarly to Lemma \ref{lem:ultcomp}, 
Lemma \ref{lem:functionookii} and  the completeness of the space  
$(C(H, [0, \infty)),  \mathcal{E}_H)$ lead to the following:
\begin{lem}\label{lem:ultcompfunc}
Let $S$ be a range set. 
Let $H$ be a topological space. 
Then  the ultrametric space 
$(C(H, S), \mathcal{U}_H^S)$ is complete. 
\end{lem}
\begin{proof}
Let 
$\{f_n\}_{n\in \nn}$ 
be a Cauchy sequence in 
$C(H, S)$. 
Then for every $x\in H$, 
we find that 
$\{f_n(x)\}_{n\in \nn}$ is Cauchy in 
$(S, u_S)$. 
By Lemma \ref{lem:ultcomp}, 
$\{f_n(x)\}_{n\in \nn}$ 
have a limit. 
Let 
$F(x)\in S$ 
be a limit of 
$\{f_n(x)\}_{n\in \nn}$. 
By Lemma \ref{lem:functionookii}, 
the sequence 
$\{f_n\}_{n\in \nn}$ is also Cauchy in 
$(C(H, [0, \infty)), \mathcal{E}_H)$, 
and it has a limit 
$G\in (C(H, [0, \infty)), \mathcal{E}_H)$. 
Note that $G$ is continuous. 
Lemma \ref{lem:ultcomp} yields $F=G$. 
Therefore 
$\{f_n\}_{n\in \nn}$ 
has a limit in 
$C(H, S)$. 
This finishes the proof. 
\end{proof}

In  the proof of Theorem \ref{thm:ultloc}, 
 to apply the intersection property of Baire spaces to  dense $G_{\delta}$ subsets, 
we need the following:
\begin{lem}\label{lem:ultBaire}
Let $S$ be a range set. 
For every second countable  locally compact $S$-valued  ultrametrizable space $X$, 
the space $\ult(X, S)$ is  a Baire space. 
\end{lem}
\begin{proof}
By Lemma \ref{lem:ultcompfunc}, 
the space 
$(C(X^2, S), \mathcal{U}_{X^2}^S)$ 
is completely metrizable. 
By Lemma \ref{lem:gdelta}, 
in order to prove the lemma, 
it suffices to show that 
$\ult(X, S)$ 
is 
$G_{\delta}$ 
in  
$(C(X^2, S), \mathcal{U}_{X^2}^S)$.

We denote by $Q$ the set of all 
$f\in C(X^2, [0, \infty))$ 
such that 
\begin{enumerate}
	\item for every $x\in X$ we have  $f(x, x)=0$;
	\item for all $x, y\in X$ we have $f(x, y)=f(y, x)$;
	\item for all $x, y, z\in X$ we have  $f(x, y)\le f(x, z)\lor f(z, y)$. 
\end{enumerate}
Namely, 
$Q$ 
is the set of all continuous pseudo-ultrametrics on $X$. 
The set  $Q$ is 
a closed subset in  the space
$(C(X^2, [0, \infty)), \mathcal{E}_{X^2})$. 
Since all closed subsets of a metric space are 
$G_{\delta}$ 
in the whole space, 
the set $Q$ is 
$G_{\delta}$ 
in the space
$(C(X^2, [0, \infty)), \mathcal{E}_{X^2})$. 
\par

Since $X$ is second countable and locally compact, 
 we can  take a sequence 
$\{D_n\}_{n\in \nn}$ 
of compact subsets of 
$X^2$ 
with 
$\bigcup_{n\in \nn}D_n=X^2\setminus \Delta_X$, 
where 
$\Delta_X$ 
is the diagonal set of $X^2$, 
and we can take a sequence 
$\{K_n\}_{n\in \nn}$ 
of compact subsets of $X$ 
with 
$K_n\subseteq \nai(K_{n+1})$ 
and 
$\bigcup_{n\in \nn} K_n=X$, 
where 
$\nai$ 
stands for  the interior operator of $X$. 

As in the proof of   \cite[Theorem 5.1]{IsI}, 
for every 
$n\in \nn$,  
let 
$L_n$ 
be the set of all 
$f\in C(X^2, [0, \infty))$ 
for which 
there exist 
$c\in (0, \infty)$ 
and 
$N\in \nn$ 
such that 
for each  
$k>N$ 
we have 
\[
	\inf_{x\in K_n}\inf_{\  y\in X\setminus K_k}f(x, y)>c. 
\]
For each 
$n\in \nn$,  
let 
$E_n$ 
be 
the set of all 
$f\in C(X^2, [0, \infty))$ 
such that
for each 
$(x, y)\in D_n$ 
we have 
$0<f(x, y)$.  
In the proof of \cite[Theorem 5.1]{IsI}, 
it was proven that 
each 
$L_n$ 
and each 
$E_{n}$ 
are open subsets of the space
$(C(X^2, [0, \infty)), \mathcal{E}_{X^2})$. 
\par

Similarly to the proof of \cite[Theorem 5.1]{IsI}, we obtain
\[
	\ult(X)=
	Q
	\cap 
	\left(
	\bigcap_{n\in \nn} L_n
	\right)
	\cap  
	\left(
	\bigcap_{n\in \nn}E_{n}
	\right)
\]
as subsets of 
$(C(X^2, [0, \infty)), \mathcal{E}_{X^2})$; 
namely, 
$\ult(X)$ is a $G_{\delta}$ 
subset of 
the space 
$(C(X^2, [0, \infty)), \mathcal{E}_{X^2})$. 
Since the inclusion map from the space $(C(X^2, S), \mathcal{U}_{X^2}^S)$ into the space $(C(X^2, [0, \infty)), \mathcal{E}_{X^2})$ is continuous (Lemma \ref{lem:functionookii}), and since 
$\ult(X, S)=\ult(X)\cap C(X^2, S)$, 
we conclude that 
$\ult(X, S)$ is $G_{\delta}$ in the space $(C(X^2, S), \mathcal{U}_{X^2}^S)$. 
This completes the proof of the lemma. 
\end{proof}

For a property $P$ on metric spaces, 
we say that a metric space 
$(X, d)$ 
satisfies the 
\emph{local $P$} 
if 
every non-empty open metric  subspace of $X$ satisfies the property $P$.

\begin{thm}\label{thm:ultloc}
Let 
 $S$ be a quasi-complete range set with the countable coinitiality. 
Let $X$ be a second countable, 
locally compact locally non-discrete $S$-valued ultrametrizable space. 
Then for every $S$-ultra-singular transmissible parameter 
$\mathfrak{G}$, 
the set of all 
$d\in \ult(X, S)$ 
for which 
$(X, d)$ 
satisfies 
the local anti-$\mathfrak{G}$-transmissible property
 is a  dense 
 $G_{\delta}$ set
 in 
 the space 
 $(\ult(X, S), \mathcal{UD}_X^S)$. 
\end{thm}

\begin{proof}
Let $S$ be a quasi-complete range set with the countable coinitiality. 
Let $X$ be a second countable,  
locally compact locally non-discrete $S$-valued  ultrametrizable space, 
and let 
$\mathfrak{G}$ 
be an  $S$-ultra-singular transmissible parameter. 
Put 
$\mathfrak{G}=(Q, P, F, G, Z, \phi)$. 
Let 
$E$ 
be the set of all 
$S$-valued ultrametrics 
$d\in \ult(X, S)$ 
for which 
$(X, d)$ 
satisfies 
the local anti-$\mathfrak{G}$-transmissible property. 
Let 
$\{U_i\}_{i\in \nn}$ 
be a countable open base of $X$,  and let 
$\{R_i\}_{i\in \nn}$ 
be a family of 
$(\omega_0+1)$-subspaces of 
$X$ with 
$R_i\subseteq U_i$. 
Since 
$\{U_i\}_{i\in \nn}$ 
is  an open base of $X$, 
by Lemma \ref{lem:here}, 
we have
\[
	E
	=
	\bigcap_{i\in \nn}
	\bigcap_{q\in Q}
	\bigcup_{z\in Z}
	\bigcup_{a\in \seq(G(q), U_i)}
	US(X, S, \mathfrak{G}, q, a, z). 
\]
Corollary  \ref{cor:ultopen} implies that 
$E$ is 
$G_{\delta}$ 
in 
$\ult(X, S)$. 
By the definitions, 
for each 
$i\in \nn$,  
the set 
\[
	\bigcap_{q\in Q}
	\bigcup_{z\in Z}
	\bigcup_{a\in \seq(G(q), U_i)}
	US(X, S, \mathfrak{G}, q, a, z)
\]
contains 
$UT(X, S, R_i, \mathfrak{G})$.  
From Proposition \ref{prop:ultdense} it follows  that 
each set 
$UT(X, S, R_i, \mathfrak{G})$ 
is dense in 
$\ult(X, S)$. 
By Lemma \ref{lem:ultBaire},  
the space 
$\ult(X, S)$ 
is a Baire space. 
Since  $E$ is an intersection of countable dense 
$G_{\delta}$  sets
in 
a Baire space $\ult(X, S)$, 
the set $E$ is dense $G_{\delta}$ in $\ult(X, S)$. 
This completes the proof. 
\end{proof}
%%%%%%%%%%%%%%%%%%%%%%%%%%%
%%%%%%%%%%%%%%%%%%%%%%%%%%%
%%%%%%%%%%%%%%%%%%%%%%%%%%%

%%%%%%%%%%%%%%%%%%%%%%%%%%%
%%%%%%%%%%%%%%%%%%%%%%%%%%%
%%%%%%%%%%%%%%%%%%%%%%%%%%%

%%%%%%%%%%%%%%%%%%%%%%%%%%%
%%%%%%%%%%%%%%%%%%%%%%%%%%%
%%%%%%%%%%%%%%%%%%%%%%%%%%%

%%%%%%%%%%%%%%%%%%%%%%%%%%%
%%%%%%%%%%%%%%%%%%%%%%%%%%%
%%%%%%%%%%%%%%%%%%%%%%%%%%%

%%%%%%%%%%%%%%%%%%%%%%%%%%%
%%%%%%%%%%%%%%%%%%%%%%%%%%%
%%%%%%%%%%%%%%%%%%%%%%%%%%%

%%%%%%%%%%%%%%%%%%%%%%%%%%%%%%%
%%%%%%%%%%%%%%%%%%%%%%%%%%%
%%%%%%%%%%%%%%%%%%%%%%%%%%%
%%%%%%%%%%%%%%%%%%%%%%%%%%%

%%%%%%%%%%%%%%%%%%%%%%%%%%%
%%%%%%%%%%%%%%%%%%%%%%%%%%%
%%%%%%%%%%%%%%%%%%%%%%%%%%%

%%%%%%%%%%%%%%%%%%%%%%%%%%%
%%%%%%%%%%%%%%%%%%%%%%%%%%%
%%%%%%%%%%%%%%%%%%%%%%%%%%%

%%%%%%%%%%%%%%%%%%%%%%%%%%%%%%%
%%%%%%%%%%%%%%%%%%%%%%%%%%%
%%%%%%%%%%%%%%%%%%%%%%%%%%%
%%%%%%%%%%%%%%%%%%%%%%%%%%%

%%%%%%%%%%%%%%%%%%%%%%%%%%%
%%%%%%%%%%%%%%%%%%%%%%%%%%%
%%%%%%%%%%%%%%%%%%%%%%%%%%%

%%%%%%%%%%%%%%%%%%%%%%%%%%%
%%%%%%%%%%%%%%%%%%%%%%%%%%%
%%%%%%%%%%%%%%%%%%%%%%%%%%%

\renewcommand{\arraystretch}{1.2}

%%%%%%%%%%%%%%%%%%%%%%%%%%%%%%%
\setlongtables
\begin{tabularx}{\textwidth}{|X|c|}
\caption{Table of notion}
\endhead
\hline 

Notion
&
Place
\\
\hline 
\hline

Ultrametrics 
& 
Section \ref{intro}\\ 
\hline

$S$-valued metrics 
& 
Section \ref{intro}
\\
\hline

Ultra-norms, ultra-normed modules
&
Section \ref{intro}
\\
\hline

$R$-independency
&
Section \ref{intro}
\\
\hline

$S$-valued ultrametrizability, 
ultrametrizability
& 
Section \ref{intro}
\\
\hline 

Complete $S$-valued ultrametrizability, 
complete ultrametrizability
& 
Section \ref{intro}
\\
\hline

Countable coinitiality 
& 
Section \ref{intro}
\\
\hline

$C$-quasi-completeness, quasi-completeness 
&
Section \ref{intro}
\\
\hline

Amenability of functions 
&
Subsection \ref{modi:ult}
\\
\hline

Invariant metrics 
&
Subsection \ref{subsec:inv}
\\
\hline

$0$-dimensional spaces, 
ultranormal spaces
& 
Subsection \ref{subsec:basic}
\\
\hline

Lower-semicontinuity of set-valued maps 
&
Subsection 
\ref{subsec:conti}
\\
\hline

Baire spaces 
&
Subsection \ref{subsec:Baire}
\\
\hline 

Eventually $o$-valued 
maps 
&
Subsection \ref{subsec:31}
\\
\hline

$(T, \le_T)$-valued ultrametric spaces
& 
Subsection \ref{subsec:genult}
\\
\hline

Transmissible parameters
&
Subsection \ref{subsec:61}
\\
\hline

The $\mathfrak{G}$-transmissible property
&
Subsection \ref{subsec:61}
\\
\hline

The anti-$\mathfrak{G}$-transmissible property
&
Subsection \ref{subsec:61}
\\
\hline

$S$-ultra-singularity
&
Subsection \ref{subsec:62}
\\
\hline

$(\omega_0+1)$-spaces
&
Subsection \ref{subsec:62}
\\
\hline

The doubling property
&
Subsection \ref{subsec:examples}
\\
\hline 

The rich $S$-ultra-pseudo-cones property
&
Subsection \ref{subsec:examples}
\\
\hline

Pseudo-cones 
&
Subsection \ref{subsec:examples}
\\
\hline

\end{tabularx}

%%%%%%%%%%%%%%%%%%%%%%%%%%%
%%%%%%%%%%%%%%%%%%%%%%%%%%%
%%%%%%%%%%%%%%%%%%%%%%%%%%%

%%%%%%%%%%%%%%%%%%%%%%%%%%%
%%%%%%%%%%%%%%%%%%%%%%%%%%%
%%%%%%%%%%%%%%%%%%%%%%%%%%%
\setlongtables
\begin{tabularx}{\textwidth}{|l|X|}
\caption{Table of symbols}
\endhead
\hline 

Symbol 
&
Description
\\
\hline
\hline 

$\lor$ & 
The maximal operator of $\rr$
\\ 
\hline
$0_R$, $0_V$ & 
The zero elements of a commutative ring $R$ and 
a module $V$,  respectively
\\
\hline

$\nn$ 
& 
The set of all positive integers
\\
\hline

$\cl(A)$ 
& 
The closure of $A$ in an ambient space of $A$
\\
\hline

$\card(A)$ 
& 
The cardinality of a set $A$
\\
\hline

$B(x, r)$, $B(x, r; d)$
& 
The closed ball centered at $x$ with radius $r$ 
in a metric space $(X, d)$
\\
\hline

$U(x, r)$, $U(x, r; d)$
& 
The open ball centered at $x$ with radius $r$ 
in a metric space $(X, d)$\\
\hline

$S_{+}$ 
& 
$S_{+}=S\setminus \{0\}$, where $S$ is a range set
\\
\hline

$\met(X)$ 
& 
The set of all metrics generating the same  topology on $X$ 
\\
\hline

$\ult(X, S)$ 
& 
The set of all $S$-valued ultrametrics  generating the same  topology 
on $X$
\\
\hline

$\ult(X)$ 
& 
The set of all ultrametrics generating the same topology on $X$ ($\ult(X)=\ult(X, [0, \infty))$)
\\
\hline

$\mathcal{D}_X$ 
& 
The metric on $\met(X)$
\\
\hline

$\mathcal{UD}_X^S$ 
& 
The ultrametric on $\ult(X, S)$
\\
\hline

$d\times_{\infty} e$ 
& 
The $\ell^{\infty}$-product metric of metrics $d$ and $e$
\\
\hline

$\mathcal{C}(Z)$ 
& 
The set of all non-empty closed subsets of a  metrizable space $Z$
\\
\hline

$\mathrm{F}(R, X, o)$ 
& 
The free $R$-module generated by $X$ 
such that  $o$ is its  zero element, 
where $o\not\in X$
\\
\hline

$\map(A, B)$ 
& 
The set of all maps from $A$ to $B$
\\
\hline

$\elel(S, M, o)$ 
& 
The set of all eventually $o$-valued maps 
from 
$S_{+}$ to $M$,  where  $S$ is a range set  and $o\in M$
\\
\hline

$\Delta$ 
& 
The ultrametric on $\elel(S, M, o)$
\\
\hline

$\tau^{*}d$ 
& 
The pullback metric 
induced from a metric $d$ and a bijection $\tau$
defined in Subsection 5.1
\\
\hline

$\mathcal{P}^{*}(\nn)$ 
& 
The set of all non-empty subsets of $\nn$
\\
\hline

$\mathcal{F}(T)$ 
& 
The set of all closed subsets of  a topological space $T$
\\
\hline

$\seq(W, E)$
& 
The set of  all finite injective sequences 
$\{a_i\}_{i=1}^n$ in $E$ with $n\in W$, where 
$W\subset \nn$
\\
\hline

$US( X, S, \mathfrak{G}, q, a, z)$ 
& 
The set of all 
$d\in \ult(X, S)$ 
such that 
$\phi^{q, X}(a, z, d)\in P\setminus F(q)$, 
where $\mathfrak{G}=(Q, P, F, G, Z, \phi)$ is a transmissible parameter
\\
\hline

$US(X, S,  \mathfrak{G})$  
& 
The set of all 
$d\in \ult(X, S)$ 
such that 
$(X, d)$ 
satisfies the anti-$\mathfrak{G}$-transmissible property
\\
\hline

$UT(X, S, R, \mathfrak{G})$  
& 
The set of all 
$d\in \ult(X, S)$
 for which 
$(R, d|_{R^2})$  satisfies the anti-$\mathfrak{G}$-transmissible 
property, where $R$ is an $(\omega_0+1)$-subspace of $X$.  
\\
\hline

$\cu(X, S)$ 
& 
The set of all complete $S$-valued ultrametrics  generating the same topology on $X$
\\
\hline

$\alpha_{d}(A)$ 
& 
$\alpha_{d}(A)=\inf\{\, d(x, y)\mid \text{$x, y\in A$ and $x\neq y$}\, \}$
\\
\hline

$\mathcal{GH}$ 
& 
The Gromov--Hausdorff distance
\\
\hline

 $\pc(X, d)$  
 & 
 The class of  all pseudo-cones of 
 $(X, d)$
 \\
 \hline

$\mathscr{U}_T$ 
&
The class of 
all finite ultrametric spaces 
on which  all distances are in a range  set $T$
\\
\hline

$C(H, S)$ 
& 
The set of all continuous functions from a topological space  $H$ into 
a range set $S$
\\
\hline

$u_S$ 
& 
The ultrametric on a range set $S$ defined by $u_S(x, y)=x\lor y$
\\
\hline

$d_E$ 
& 
The Euclidean metric on a range set $S$
\\
\hline

$\mathcal{U}_H^S$ 
& 
The ultrametric on $C(H, S)$ defined by 
$\mathcal{U}_H^S(f, g)=\min\left\{ 1,\sup_{x\in H}u_S(f(x), g(x)) \right\}$ 
\\
\hline

$\mathcal{E}_H$ 
& 
The metric on $C(H, [0, \infty))$ defined by 
$\mathcal{E}_H(f, g)=\min\left\{ 1,\sup_{x\in H}|f(x)-g(x)| \right\}$ 
\\
\hline

\end{tabularx}

%%%%%%%%%%%%%%%%%%%%%%%%%%%
%%%%%%%%%%%%%%%%%%%%%%%%%%%
%%%%%%%%%%%%%%%%%%%%%%%%%%%

\begin{ac}
The author would like to thank Professor Koichi Nagano for his advice and constant encouragement. 
The author  would also  like to thank the referee for helpful comments and suggestions. 
\end{ac}

%%%%%%%%%%%%%%%%%%%%%%%%%%%%%%%
%%%%%%%%%%%%%%%%%%%%%%%%%%%%%%%%%%%%%%
%%%%%%%%%%%%%%%%%%%%%%%%%%%%%%%%%%%

This preprint of the Work accepted for publication in 
\textit{$p$-Adic Numbers, Ultrametric Analysis and Applications}, 
$\copyright$, copyright (2021), Pleiades Publishing, Ltd.; \\
\verb|https://www.pleiades.online/en/journal/pultraan/#prettyPhoto|

\end{document}